\documentclass[12pt,onecolumn]{IEEEtran}

\usepackage{ulem,mathrsfs,xcolor,algorithm2e,amssymb,amsfonts,graphicx,amsmath,epstopdf,array,subfigure,authblk,amsthm,chemarr}
\newcommand{\lra}{\mathop{\longrightarrow}\limits}
\newcommand{\R}{{\mathrm R}}

\def\undertilde#1{\mathord{\vtop{\ialign{##\crcr
				$\hfil\displaystyle{#1}\hfil$\crcr\noalign{\kern1.5pt\nointerlineskip}
				$\hfil\tilde{}\hfil$\crcr\noalign{\kern1.5pt}}}}}
\newtheorem{corollary}{Corollary}
\newtheorem{lemma}{Lemma}

\newtheorem{definition}{Definition}

\newtheorem{remark}{Remark}

\newtheorem{theo}{Theorem}
 \renewcommand{\det}{\mbox{det}}
\newcommand{\beq}{\begin{equation}}
\newcommand{\eeq}{\end{equation}}

	 \renewcommand{\emph}[1]{\textit{#1}}

\title{A robust Lyapunov criterion for non-oscillatory behaviors in biological interaction networks}
\author{  David Angeli, M. Ali Al-Radhawi and  Eduardo D. Sontag 
\thanks{D. Angeli is with the  Department of Electrical \& Electronic Engineering, Imperial College London, London SW7 2AZ, UK. He is also with Dipartimento di Ingegneria dell'Informazione, University of Florence, Florence, Italy. Email: \texttt{d.angeli@imperial.ac.uk}}
\thanks{M. Ali Al-Radhawi and E. D. Sontag are with Departments of Electrical and Computer Engineering and of Bioengineering, Northeastern University, Boston, MA 02115, USA.  E. D. Sontag is also an affiliate of the Departments of Chemical Engineering and Mathematics at Northeastern University, and he is with  the Laboratory of Systems Pharmacology, Harvard Medical School, Boston, MA 02115, USA. Emails:  \texttt{malirdwi@northeastern.edu}, \texttt{e.sontag@northeastern.edu}.} 
\thanks{This work was partially supported by grants ONR N00014-21-1-2431 and AFOSR FA9550-21-1-0289.}}

\begin{document}
 	\maketitle
 	\begin{abstract}
 		We introduce the notion of non-oscillation, propose a constructive method for its robust verification, and study its application to biological interaction networks (also known as, chemical reaction networks). We begin by revisiting Muldowney's result on the non-existence of periodic solutions based on the study of the variational system of the second additive compound of the Jacobian of a nonlinear system. We show that exponential stability of the latter rules out limit cycles, quasi-periodic solutions, and broad classes of oscillatory behavior. We focus then on nonlinear equations arising in biological interaction networks with general kinetics, and we show that  the dynamics of the aforementioned variational system can be embedded in a linear differential inclusion. We then propose algorithms for constructing piecewise linear Lyapunov functions to certify global robust  non-oscillatory behavior. Finally, we apply our techniques to study several regulated enzymatic cycles where available methods are not able to provide any information about their qualitative global behavior.  
 	\end{abstract}

 	\section{Introduction}
 	Natural and engineered nonlinear systems are commonly required to operate consistently and robustly under perturbations and a variety of  environmental conditions. Rational analysis and synthesis of such systems need qualitative characterizations of  their  global long-term behavior, which is a notoriously difficult task for general nonlinear systems.  This problem is compounded by the large uncertainties that pervade the mathematical descriptions of many such systems. A prominent class  exemplifying these difficulties are  biological interaction networks, which include molecular processes such as expression and decay of proteins, metabolic networks, regulation of transcription and translation, and signal transduction \cite{alon}. Such networks are usually described via the mathematical formalism of Biological Interaction Networks (BINs) (also known as Chemical Reaction Networks (CRNs)) \cite{erdi}. Ordinary Differential Equations (ODE) descriptions of BINs have two components, one  graphical and one kinetic. The first is often well-characterized as it corresponds to the list of reactions, while the latter (which includes kinetic constants and the functional forms of kinetics) is not, as it depends on quantifying the ``speed'' of reactions which is difficult to measure and subject to environmental changes. This information disparity precludes the construction of full mathematical models, and hence a pressing need has emerged  for the development of general \emph{robust} techniques that can provide conclusions on the qualitative behavior of the network based on the graphical information only \cite{bailey01}.  
 	
 	Although this problem may seem intractable, significant progress has been made in the past few decades. A pioneering example has been the development of the theory of complex-balanced networks with Mass-Action kinetics, and the associated deficiency-based characterizations \cite{horn72,feinberg87}. It has been shown that such networks always admit Lyapunov functions over the positive orthant, and that global stability can be ascertained in some cases \cite{sontag01,anderson}.  Other notions of global behavior have also been  considered in the literature. It has been shown that the \emph{persistence} of a class of BINs can be certified via simple graphical conditions \cite{persistence}. The \emph{monotonicity} of certain BINs can be established in reaction coordinates, and this property has been used to show global convergence to attractors \cite{angeli10}.  More recently, new techniques have been developed for certifying global stability by the construction of Robust Lyapunov Functions (RLFs) in reaction coordinates \cite{MA_cdc13,MA_TAC,plos} and concentration coordinates \cite{MA_cdc14,plos,blanchini,blanchini2}. These techniques have been developed into a comprehensive framework with relatively wide applicability to various key biochemical networks like transcriptional networks, post-translational modification cascades, signal transduction, etc \cite{plos}.
 	
 	Despite recent advances, many relevant networks, and many dynamic behaviors,  remain outside the scope of analysis through available methods. In this paper, we study oscillations in dynamical systems with particular emphasis on BINs. Unlike earlier works which studied conditions for the emergence of oscillations  in various physical contexts \cite{angeliosc08,elwakil10}, we propose to study  another global qualitative notion, which we call \emph{non-oscillation}, by examining   the variational system of the second additive compound of the Jacobian of a nonlinear system. This approach was originally introduced in order to rule out periodic solutions by Muldowney \cite{muldowney} (see also  \cite{margaliot}, where the approach has recently been reframed in the context of $k$-Order Contraction Theory), and it has been applied to the study of epidemic models \cite{li95}, circadian rhythms \cite{wang10}, and, most remarkably, as a local analysis tool, \cite{pantea14}, to rule out Hopf bifurcations in BINs. We begin by revisiting Muldowney's results. We will show that exponential stability of the aforementioned variational system guarantees that the area measure of all bidimensional compact surfaces asymptotically converges to zero. It turns out, as a consequence, that the same will be true of the $k$th-hypervolume measure for arbitrary $k$-dimensional submanifolds for any $k\ge 2$. This allows us to exclude limit cycles, invariant torii, (asymptotically) quasi-periodic solutions, and many types of oscillatory behavior. We then show that this notion can be verified successfully for  classes of BINs where no other technique has proved useful.  We will achieve this goal by embedding the dynamics of the second additive compounds of a BIN in a linear differential inclusion,  and then generalize the RLF approach to be applied to this LDI. We will show that the existence of such an RLF will guarantee robust non-oscillation by establishing a LaSalle-like condition.  %
 	
 	Although robust non-oscillation is technically weaker than global stability, coupling it with local asymptotic stability is  \emph{nearly as good} as it places robust and strong constraints on the  {range} of possible behaviors of a given network. Furthermore, this new notion is also compatible with multi-stability and almost global stability \cite{angeli04,efimov12}, which opens the door for applications to systems with multiple attractors.

 	\subsection{Motivating example: regulation of the enzymatic cycle }
 	\begin{figure}[t!]
 		\centering
 		\subfigure[]{\includegraphics[height=2in]{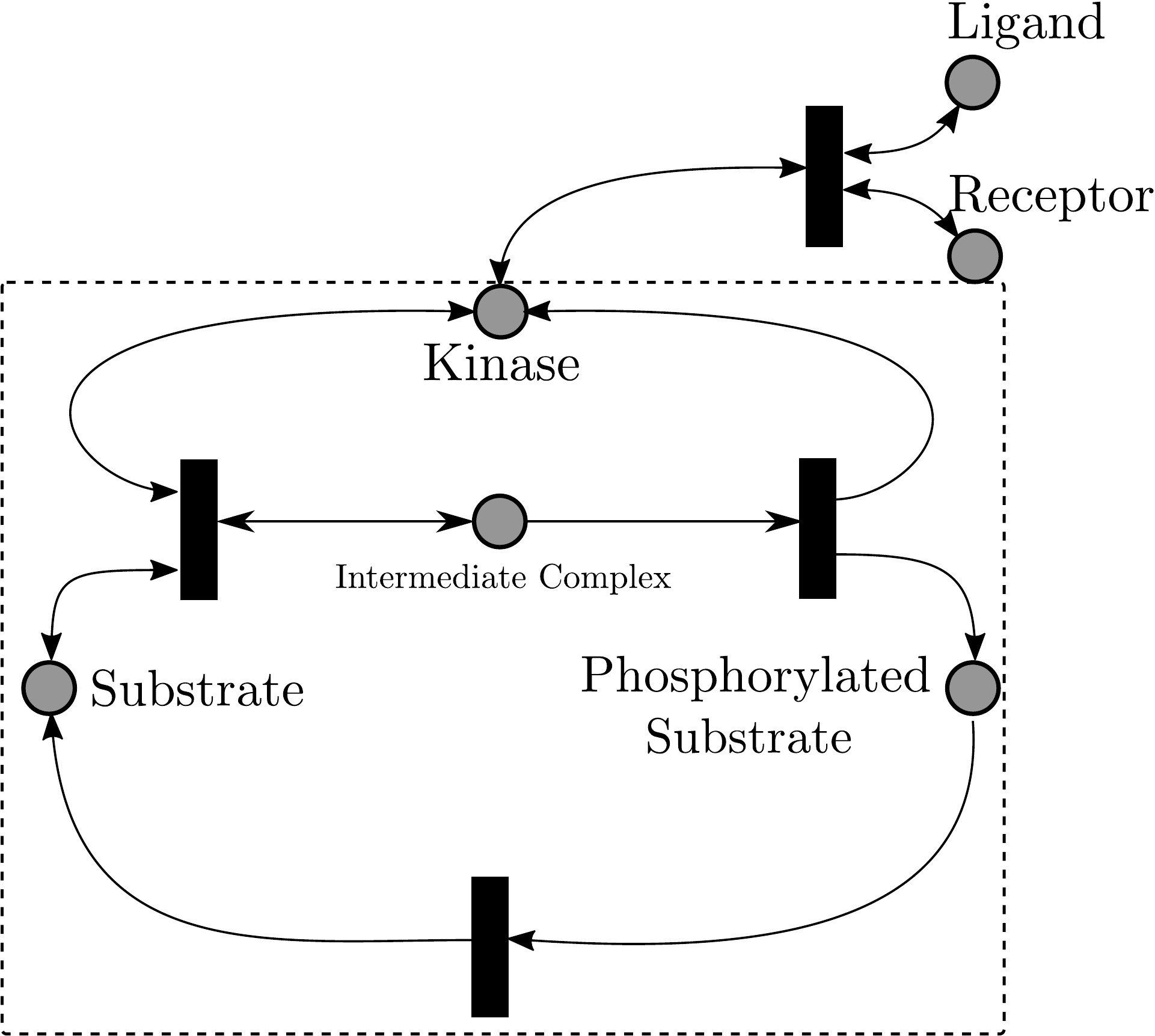}}
 		\subfigure[]{\includegraphics[width=2in]{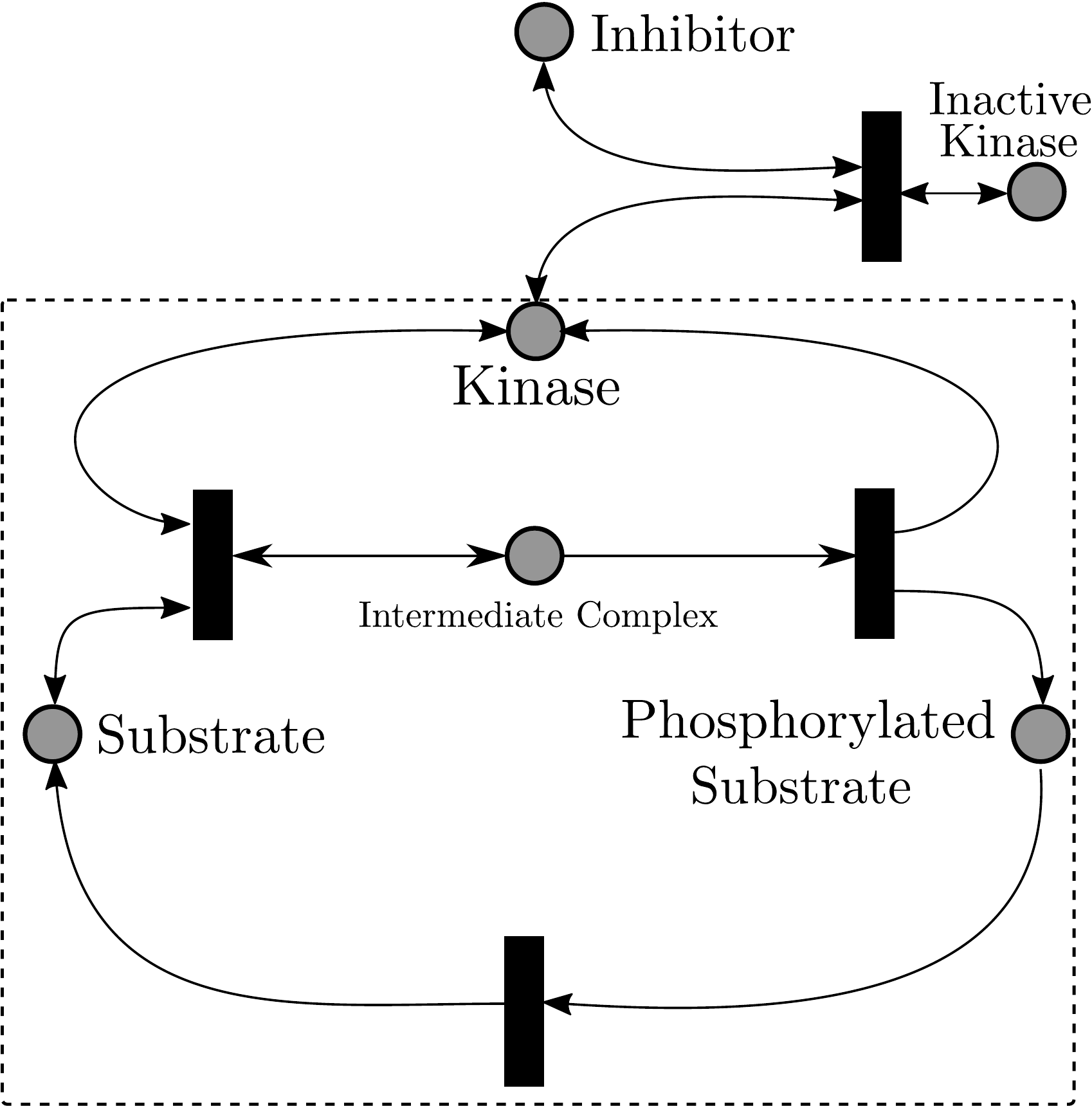}}
 		\subfigure[]{\includegraphics[width=2in]{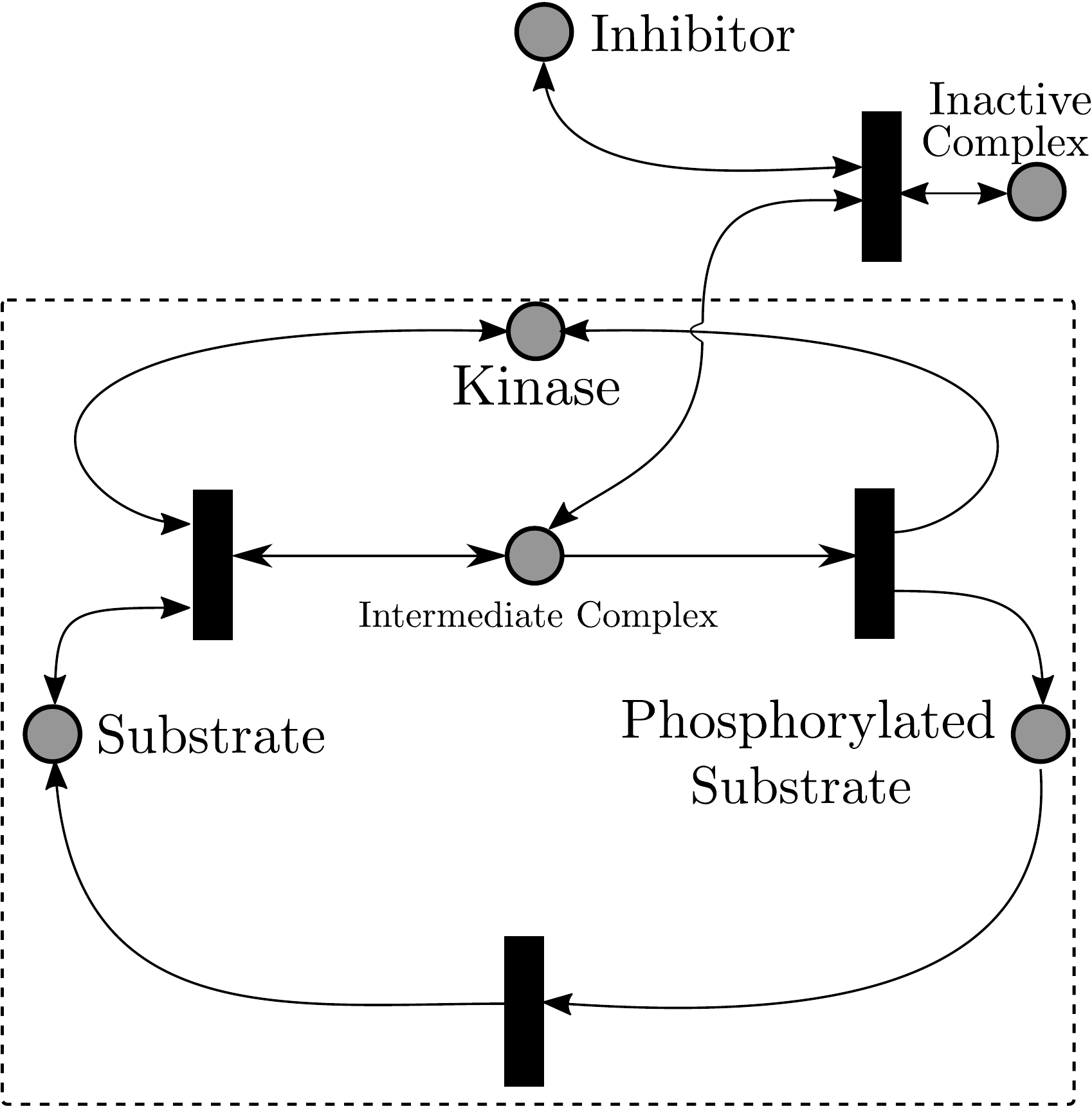}}
 		\caption{\textbf{Various architectures for regulating the PTM cycle}. (a) The kinase is only activated if a ligand binds to a receptor, (b) The kinase gets deactivated after binding to an inhibitor, (c) The substrate-kinase intermediate complex gets sequestered by an inhibitor. The Petri-net \cite{petri08} notation is used where a circle denotes a species, and a rectangle denotes a reaction.  }
 		\label{f.ptm}
 	\end{figure}
 	We describe an open problem which is highly relevant to systems biology. It involves regulation mechanisms of  the Post-Translational Modification (PTM) cycle which is a very common motif in signal transduction \cite{vecchio_murray}. For example, an enzyme known  as a kinase ($K$) binds to a substrate ($S$) to form an intermediate complex ($C$). Then, the substrate is phosphorylated to produce an activated substrate ($P$).  The activated substrate decays back to its inactive form ($S$).  The network is depicted inside the dashed rectangle in Figure \ref{f.ptm}-a)-c), and it can be written as follows:
 	\begin{equation}\label{ptm}
 		S+K \rightleftharpoons C \longrightarrow P+K, \ P \longrightarrow S. 
 	\end{equation}
 	The dynamics of the above network has been analyzed using a Piecewise Linear (PWL) RLF. In particular, it has been shown that it always admits a positive globally asymptotic stable steady state, for any choice of monotone kinetics \cite{MA_TAC,plos}. 
 	
 	However, small structural changes in the network can make a PWL RLF fail to exist. We study various ways of regulating the activity of the cycle as depicted in Figure \ref{f.ptm}. In one scenario, the kinase can only be activated if two molecules bind (e.g, a ligand ($L$) and a receptor  {($Rc$))} as shown in Figure \ref{f.ptm}-a. This is modelled by adding the reaction \begin{equation} \label{ptmRL}  {Rc}+L   \rightleftharpoons K\end{equation} to the BIN \eqref{ptm}. It can be shown that this network  has a unique positive steady steady state for each assignment of non-zero total substrate, ligand and receptor concentrations \cite{craciun05,craciun06}. However, a PWL RLF fails to exist \cite{dissertation,plos}. It has been shown recently that this network enjoys \emph{local} asymptotic stability for any choice of kinetics, i.e., the Jacobian matrix is always Hurwitz at any steady state \cite{colaneri}. However, there are no known robust global guarantees on the asymptotic behavior.  Other regulation mechanisms exist \cite{craciun06}. For instance, the kinase might be inactivated by binding to an inhibitor ($I$) such as a drug used in targeted cancer therapies \cite{bhullar18}. This is represented by adding the reaction $K+I \rightleftharpoons EI$ to the network \eqref{ptm} as shown in Figure \ref{f.ptm}-b). A third possible architecture has the intermediate complex ($C$) sequestered by $I$. Hence, the reaction $C+I \rightleftharpoons CI$ is added to \eqref{ptm}. None of these networks can be globally analyzed using current techniques.    We will be studying these networks under our new framework and show that they are globally non-oscillatory.
 	
 	It is worth mentioning that not all regulation mechanisms of the PTM are beyond current methods of analysis. For instance,  instead of a simple decay of $P$ to $S$, another enzyme called a phosphatase can be used  to accelerate the dephosphorylation of $P$ back to $S$. This latter architecture is well-studied \cite{angeli10}, and its global stability can be certified by a PWL RLF \cite{plos}.
 	
 	This paper is organized as follows. Mathematical definitions and notation are given in section II. Section III revisits Muldowney's results in terms of exponential stability. Section 4 provides a robust Lyapunov criterion for robust non-oscillation when the dynamics can be embedded in a Linear Differential Inclusion (LDI). Section 5 studies the application of the results to BINs. Section 6 provides algorithms for constructing the PWL RLF. Section 7 studies several examples of enzymatic cycles that have not been amenable to methods in the literature. Finally, section 8 is dedicated to a brief discussion of the results.

 	\section{Non-oscillatory systems}
 	\subsection{Definitions and Notation}
 	Our basic concepts and results are not restricted to BINs, but apply to more general classes of nonlinear systems.
 	For a dynamical system 
 	\begin{equation} 
 		\label{thesys}
 		\dot{x}(t) = f(x(t)),
 	\end{equation}
 	with  {the state $x:\mathbb R_{\ge 0 } \to \mathbb{R}^n$} and $f: X \subset \mathbb{R}^n \rightarrow \mathbb{R}^n$ of class $\mathcal{C}^1$, we denote by $\varphi(t,x_0)$ the solution at time $t$ from initial condition $x_0$ at time $0$.
 	Moreover, $\omega(x_0)$ denotes the $\omega$-limit set of such a solution.
 	The set $X$ can be arbitrary, but we assume that it is forward invariant for the dynamics, that is, $\varphi(t,x_0) \in X$ for all
 	$t \geq 0$ and all $x_0 \in X$. Class $\mathcal{C}^1$ means that $f$ is the restriction of a $\mathcal{C}^1$ function defined on some open neighborhood of $X$.
 	We let $\mathbb{D} := \{ z \in \mathbb{R}^2: z_1^2 + z_2^2 \leq 1 \} \subset \mathbb{R}^2$ denote the unit disk, $\mathbb{S} := \{ [ \cos ( \theta), \sin ( \theta) ], \theta \in [0, 2 \pi] \}$ the unit circle, and $\mathbb{S}^k$ the $k$-dimensional torus.
 	\begin{definition}
 		\label{oscillatoryis}
 		We say that (\ref{thesys}) exhibits oscillatory behavior if, for some integer $k \geq 1$, it admits a compact invariant set $\Omega \subset {X}$ which is the image of a $\mathcal{C}^1$ injection $h: \mathbb{S}^k \rightarrow X$ not everywhere singular.  If it does not admit such a set then we say that \eqref{thesys} is  {non-oscillatory}.
 	\end{definition}
 	Notice that Definition \ref{oscillatoryis} includes systems with many kind {s} of  {non-converging} behavior, in particular, systems with periodic solutions, or asymptotically periodic solutions. In this case $\omega ( x_0)$ is invariant and diffeomorphic to $\mathbb{S}$.
 	Furthermore, it includes systems with multiple incommensurable oscillation frequencies, 
 	(such as quasiperiodic solutions, or asymptotically quasiperiodic solutions). In such a case $\omega(x_0)$ is the image of $\mathbb{S}^k$, for some $k>1$  and some map $h$.
 	It also includes other types of non-convergent behaviors, such as solutions approaching
 	a closed curve of equilibria, and certain types of homoclinic and heteroclinic orbits (of finite length). Moreover, it also encompasses certain types of chaotic systems as the associated attractors are sometimes known to embed unstable periodic solutions \cite{so07}.  
 	
 	 {While the gap between non-convergent and oscillatory behavior seems to be extremely small in practice, ruling out its existence appears to be very challenging, given the existing technical tools.}
 	
 	We introduce some of the required background on compound matrices and their role in assessing the evolution of $k$-hypervolumes along solutions of a dynamical system.
 	For an arbitrary $\mathcal{C}^1$ injection $h: \mathbb{D} \rightarrow X \subset \mathbb{R}^n$, the area of $h ( \mathbb{D} )$ can be computed as:
 	\begin{equation}
 		\label{areadisc}
 		\mu_2 ( h(\mathbb{D})) := \int_{\mathbb{D}}  \sqrt{ \sum_{I \subset \{1,\ldots, n\}: |I|=2 } \left [ \textrm{det} \left ( \frac{ \partial h_I }{ \partial z} (z) \right ) \right]^2  } \, dz_1 dz_2.  
 	\end{equation}  
 	where, for a set $I= \{i_1, i_2, \ldots, i_{|I|} \} \subset \{ 1,2, \ldots n \}$ with elements ordered according to  $i_1<i_2< \ldots< i_{|I|}$, and a vector $h$, $h_I$ denotes the sub-vector $[h_{i_1}, h_{i_2}, \ldots, h_{i_{|I|}}]'$.
 	Similarly, for any given $\mathcal{C}^1$ injective map $h: \mathbb{S}^k \rightarrow \mathbb{R}^n$, and $k \geq 1$, the $k$-hypervolume of $h ( \mathbb{S}^k )$ can be obtained according to:
 	\begin{align}
 		\label{areatorus}
 		&\mu_k ( h(\mathbb{S}^k)) \\ &\nonumber  := {\Huge  \int_{\mathbb{S}^k} }   \sqrt{ \sum_{I \subset \{1,\ldots, n\}:|I|=k }  \left [ \textrm{det} \left ( \frac{ \partial h_I }{ \partial \theta} (\theta) \right ) \right ]^2  } \, d\theta_1 d \theta_2 \ldots d \theta_k.  
 	\end{align}
 	These quantities can further be defined along solutions of (\ref{thesys}); in particular,
 	we aim at quantifying $\mu_2 ( \varphi(t,h( \mathbb{D}) ) )$ and $\mu_k ( \varphi(t, h(\mathbb{S}^k )))$. To this end, we associate to system (\ref{thesys})
 	the family of variational equations:
 	\begin{equation} 
 		\label{kvar}
 		\begin{array}{rcl}
 			\dot{x} &=& f(x) \\
 			\dot{\delta}^{(k)} (t) &=& { \dfrac{ \partial f}{\partial x}^{(k)} (x(t)) } \normalsize \, \delta^{(k)} (t)
 		\end{array}
 	\end{equation}
 	where $\delta^{(k)}$ is a vector in $\mathbb{R}^{{n \choose k}}$ and, for any $A \in \mathbb{R}^{n \times n}$,
 	$A^{(k)} \in \mathbb{R}^{{n \choose k}\times {n \choose k}}$ denotes the $k$-th additive compound matrices for $k=1 \ldots n$, which are defined element-wise as follows \cite{muldowney}:
 	\begin{equation}\label{e.compound}
 		A_{IJ}^{(k)}=\left\{\begin{array}{ll} A_{i_1i_1}+...+A_{i_k i_k}, & \mbox{if}~I=J \\
 			(-1)^{\ell+s} A_{i_s j_{\ell}}, & \mbox{if exactly one entry}~i_s~\mbox{of}~I\\&\mbox{does not occur in}~J\\&\mbox{and}~j_\ell~\mbox{does not occur in}~I \\ 0, & \mbox{if}~I~\mbox{differs from}~J~\mbox{in two} \\&\mbox{ or more entries},\end{array} \right.
 	\end{equation}
 	where $I,J \subset \{1,..,n\}$ are of cardinality $k$, respectively denoted as $I=\{i_1, i_2, \ldots ,i_k\}$, $J=\{ j_1, j_2, \ldots, j_k \}$ with entries indexed such that  $1 \le i_1 <i_2<...<i_k \le n$ and $1 \le j_1 < j_2 < \ldots < j_k \leq n$. \\ 

 	 {
 		To exemplify this construction, consider the case $k=2$, which will later be our main object of study, and the $4 \times 4$ matrix, $A=[a_{ij}]$. The corresponding $6 \times 6$ additive compound matrix $A^{(2)}$ reads:
 		\tiny
 		\[  \left(\begin{array}{cccccc} a_{11}+a_{22} & a_{23} & a_{24} & -a_{13} & -a_{14} & 0\\ \\ a_{32} & a_{11}+a_{33} & a_{34} & a_{12} & 0 & -a_{14}\\ \\a_{42} & a_{43} & a_{11}+a_{44} & 0 & a_{12} & a_{13}\\ \\-a_{31} & a_{21} & 0 & a_{22}+a_{33} & a_{34} & -a_{24}\\ \\-a_{41} & 0 & a_{21} & a_{43} & a_{22}+a_{44} & a_{23}\\ \\ 0 & -a_{41} & a_{31} & -a_{42} & a_{32} & a_{33}+a_{44} \end{array}\right).  \]
 	}
 	\normalsize
 	Fix any subset $J \subset \{ 1,2, \ldots, n \}$ of cardinality $k$. It is  known \cite{muldowney} that, by arranging minors of $\partial \varphi / \partial x_J$
 	for all subsets $I \subset \{1,2,\ldots,n \}$ of cardinality $k$ in lexicographic order within the vector $\delta^{(k)} (t)$ as follows
 	\begin{equation} 
 		\label{stackdet} \delta^{(k)} (t):= \left [ \begin{array}{c} \vdots \\ \textrm{det} \left ( \frac{ \partial \varphi_I }{\partial x_{J} }   (t,x) \right )  \\ \vdots \end{array} \right ],  
 	\end{equation} 
 	the resulting vector $\delta^{(k)} (t)$ fulfills the $k$-th variational equation (\ref{kvar})
 	with initial condition $x(0)=x$ and 
 	\[ \delta^{(k)} (0)= \left [ \begin{array}{c} \vdots \\ \delta_{I,J} \\ \vdots \end{array} \right ], \]
 	where $\delta_{I,J}:=1$ iff $J=I$ and $0$ otherwise. These properties will be exploited in subsequent sections to 
 	quantify how the hypervolumes previously defined evolve along solutions of the considered system of differential equations.
 	
 	\subsection{Muldowney's result revisited}
 	Our main goal for this section is to obtain an analog to Muldowney's result \cite{muldowney} by making use of the notion of uniform exponential stability.
 	His seminal paper shows that if the logarithmic norm of the second-additive compound of the Jacobian matrix is negative throughout state-space for a nonlinear system, (non-trivial) periodic solutions cannot exist.
 	We formulate the result by using the notion of uniform exponential stability of the associated second-additive compound variational equation, so that
 	we can verify assumptions and certify properties through the construction of suitable Lyapunov functions for an associated LDI. Moreover, we strengthen the original result by generalising its applicability to invariant submanifolds of any dimension. We start with the following Lemma about time varying-matrices:
 	\begin{lemma}
 		Let $\Lambda(t):\mathbb R_{\ge 0} \to \mathbb R^{n \times n}$ be a time-varying matrix. If all minors of order $k$ of $\Lambda(t)$ converge  to $0$ so do all minors of order $q\ge k$. Furthermore, if the assumed convergence is exponential, then so is the convergence of all minors of order $q\ge k$.
 	\end{lemma}
 	\begin{proof}
 		We prove the result by induction, by showing that if the convergence happens for $k$, then it is also fulfilled for $q = k+1$.  
 		
 		Recall that for an invertible square matrix $A$ of dimension $q$, it holds that
 		$A  \, \textrm{adj} (A) = \textrm{det}(A) I_q$,
 		where $\textrm{adj}(A)$ denotes the adjoint matrix of $A$. Hence, taking determinants in both sides of this previous equality we get:
 		\begin{align*} \textrm{det} (A) \cdot \textrm{det} ( \textrm{adj}(A)) &=
 			\textrm{det} ( A \, \textrm{adj} (A) ) \\ & = \textrm{det} ( \textrm{det}(A) I_q ) = \textrm{det} (A)^q.  \end{align*}
 		In particular then, $\textrm{det} ( \textrm{adj}(A)) = \textrm{det}(A)^{q-1}$. 
 		Taking absolute values and inverting this relationship yields:
 		\begin{equation}
 			\label{exppreservation}
 			|\textrm{det}(A)| = \psi( \textrm{adj} (A) ) 
 		\end{equation}	
 		where $\psi:\mathbb R^{n \times n} \to \mathbb R_{\ge 0}$ is continuous and  given as $\psi(B)= \sqrt[q-1]{ | \textrm{det} ( B ) |}$. Note that $\psi(0)=0$.
 		More generally, if $A$ is singular, then $\det(A)=0$ means also that the inequality trivially holds:
 		\begin{equation}
 			\label{boundonarea}
 			|\textrm{det}(A)| \leq \psi( \textrm{adj} (A) ).
 		\end{equation}
 		We will apply this observation  to the matrices $A= [\Lambda]_{IJ}$ for any choice of $I,J \subset \{ 1,2, \ldots, n\}$ of cardinality $q$.
 		By the induction hypothesis for any $\tilde{I}$,$\tilde{J}$ of cardinality $k$ it holds,
 		\[ \lim_{t \rightarrow + \infty }  \textrm{det} \left (  [\Lambda]_{\tilde{I}{\tilde{J} } }(t)   \right ) = 0.   \]
 		Hence, the same is true of each of the entry of the adjoint matrix (which by definition are minors of dimension $q-1=k$ possibly multiplied by $-1$):
 		\[ \lim_{t \rightarrow + \infty}   \textrm{adj}  \left (  [\Lambda]_{{I}{{J} } }(t)   \right ) = 0. \]
 		In particular, then, our convergence claim follows from (\ref{boundonarea}) and
 		continuity of $\psi$ and the fact that $\psi(0)=0$. 

 		In order to prove exponential convergence, assume that for some $M$ and $\lambda>0$, the following is true:
 		\[  \textrm{det} \left (  [\Lambda]_{\tilde{I}{\tilde{J} } }(t)   \right ) \leq M e^{- \lambda t}  \qquad \forall \, t \geq 0\]
 		for all $\tilde{I}$, $\tilde{J}$ of cardinality $k$. We see that
 		for all $I$, $J$ of cardinality $k+1$, and all $i,j$ in $\{1, \ldots, k+1\}$,	it holds that:
 		\[ \left | \textrm{adj}  \left (  [\Lambda_{IJ}](t) \right )_{i,j} \right | \leq M e^{- \lambda t} \qquad \forall \, t \geq 0. \]
 		Hence, substituting the above entry-wise upper-bound in (\ref{boundonarea}) yields: 
 		\begin{align*} | \det \left ( [\Lambda_{IJ}](t)  \right ) |  & \leq \psi \left ( \textrm{adj}  \left ([\Lambda_{IJ}](t) \right ) \right ) \\ &\leq \sqrt[k]{(k+1)! M^{k+1} e^{- (k+1) \lambda t}} = \tilde{M} e^{- \frac{k+1}{k} \lambda t} \end{align*}
 		for a suitable choice of $\tilde{M}$. This completes the proof of the induction step in the case of exponential convergence.
 	\end{proof}
 	The following corollary follows:
 	\begin{corollary}
 		\label{usefullemma}
 		Assume that for some initial condition $x(0)=x$ and some $k \in \{ 1,2, \ldots n \}$, the solutions of (\ref{kvar}) with arbitrary initial conditions
 		$\delta^{(k)}(0) \in \mathbb{R}^{{n \choose k}}$ fulfil:
 		\[ \lim_{t \rightarrow + \infty} \delta^{(k)} (t) = 0. \]
 		Then, the same is true for all solutions
 		$\delta^{(q)} (t)$ of (\ref{kvar}) for $q$ in $\{ k, k+1, \ldots, n \}$.
 		Moreover, if the assumed convergence to $0$ is exponential (and uniform), so is it for $\delta^{(q)} (t)$.
 	\end{corollary}
 	\begin{proof}
 		The result follows from the previous Lemma because of the connection between solutions of (\ref{kvar}) and minors of
 		$\frac{ \partial \varphi}{\partial x} (t,x)$ for any given initial condition  {$x_0$}. Hence, the claim is  equivalent to showing that if all minors of order $k$ of $\Lambda(t):=\frac{ \partial \varphi}{\partial x} (t,x)$ converge (exponentially) to $0$ so do all minors of order $k+1$. The latter statement immediately follows from the previous Lemma.
 	\end{proof}

 	Our main Theorem for general systems is as follows.
 	\begin{theo}
 		\label{muld}
 		Consider a dynamical system as in (\ref{thesys}):
 		 {\begin{equation}\label{thesys0}
 				\dot{x} (t) = f(x(t)),
 		\end{equation}}
 		and assume that, for some convex set $K$ and all $x_0 \in K \subset X$, the second variational equation
 		\begin{equation} 
 			\label{2var}
 			 {\dot{\delta^{(2)}} (t) = \frac{ \partial f}{\partial x}^{(2)} (x(t)) \, \delta^{(2)} (t) }
 		\end{equation}
 		is uniformly exponentially stable, i.e., there exist  $M, \lambda>0$ such that, for all $t \geq 0$ and all $\delta^{(2)} (0)$:
 		\begin{equation}
 			\label{expstability}
 			| \delta^{(2)} (t) | \leq M e^{- \lambda t} \, | \delta^{(2)} (0) |, 
 		\end{equation}
 		with $M$ and $\lambda$ independent of $x(0)$ and $\delta^{(2)}(0)$.
 		Then, the dynamical system \eqref{thesys} is non-oscillatory.
 	\end{theo}
 	\begin{proof}
 		The statement can be proved by contradiction. We start for the sake of simplicity, from the case $k=2$. %
 		Assume $h$ be a $\mathcal{C}^1$ function
 		$h : \mathbb{S}^2 \rightarrow K$, not everywhere singular, such that $h(\mathbb{S}^2) \subset K$ is invariant.
 		Of course, since $\varphi(t,h( \mathbb{S}^2)) = h ( \mathbb{S}^2 )$ by definition of invariant set:
 		\begin{equation}
 			\label{tobecontradicted}
 			\mu_2 ( \varphi(t, h( \mathbb{S}^2))) = \mu_2 (h(\mathbb{S}^2))>0, 
 		\end{equation}
 		where the last inequality follows by the implicit function theorem given that $h$ is not everywhere singular.
 		On the other hand,
 		\begin{align}
 			\label{areaforbound}
 			&\mu_2 ( \varphi(t,h(\mathbb{S}^2)) ) = \\ \nonumber &   {\Huge  \int_{\mathbb{S}^2} }   \sqrt{ \sum_{I \subset \{1,\ldots, n\}:|I|=2 }  \left [ \textrm{det} \left ( \frac{ \partial}{\partial \theta} \varphi_I (t, h(\theta))  \right ) \right ]^2  } \, d\theta_1 d \theta_2.  
 		\end{align}	
 		Moreover, by the chain rule, 
 		\[ \frac{ \partial}{\partial \theta} \varphi_I (t, h(\theta))
 		= \frac{ \partial}{\partial x} \varphi_I (t, h(\theta)) \frac{ \partial h}{\partial \theta }, \]
 		and therefore by the Cauchy-Binet formula:
 		\begin{align*} &\textrm{det} \left (   \frac{ \partial}{\partial \theta} \varphi_I (t, h(\theta))         \right ) \\   &= \sum_{J \subset{1,\ldots,n}: |J|=2 }
 			\textrm{det} \left ( \frac{ \partial}{\partial x_J} \varphi_I (t, h(\theta)) \right )   \textrm{det} \left ( \frac{ \partial h_J}{\partial \theta } \right ) \\
 			&= \sum_{J \subset{1,\ldots,n}: |J|=2 } \delta_I^{(2)} (t, [h(\theta), e_J]  ) \cdot \textrm{det} \left ( \frac{ \partial h_J}{\partial \theta } \right ).
 		\end{align*}	
 		where $\delta^{(2)} (t, [x_0,\delta_0^{(2)}])$ denotes the $\delta^{(2)}$-component of the solution of (\ref{2var}) from initial conditions $x(0)=x_0$ and $\delta^{(2)} (0) = \delta_0^{(2)}$.
 		We may therefore seek to bound from above the integrand (\ref{areaforbound}) using:
 		\begin{align*} & \textrm{det} \left (   \frac{ \partial}{\partial \theta} \varphi_I (t, h(\theta))         \right )^2 \\ & \leq 2 \sum_{J \subset{1,\ldots,n}: |J|=2 } \left [ \delta_I^{(2)} (t, [h(\theta), e_J]  ) \right ]^2 \cdot \textrm{det} \left ( \frac{ \partial h_J}{\partial \theta } \right )^2. \end{align*}
 		Taking sums over $I$ we get:
 		\begin{align*}
 			&\sum_{I \subset \{1,\ldots, n\}:|I|=2 }  \left [\textrm{det} \left ( \frac{ \partial}{\partial \theta} \varphi_I (t, h(\theta))  \right )\right ]^2     
 			\\&  \leq 2 \sum_{I,J \subset{1,\ldots,n}: |I|,|J|=2 } \left [ \delta_I^{(2)} (t, [h(\theta), e_J]  ) \right ]^2 \cdot \left [\textrm{det} \left ( \frac{ \partial h_J}{\partial \theta } \right )\right ]^2  \\
 			& =  2 \sum_{J \subset{1,\ldots,n}: |J|=2 } \left | \delta^{(2)} (t, [h(\theta), e_J]  ) \right |^2 \cdot \left [ \textrm{det} \left ( \frac{ \partial h_J}{\partial \theta } \right )\right ]^2.  \end{align*}
 		Moreover, by exponential uniform stability of (\ref{2var}) we see that:
 		\begin{align*}  &\mu_2( \varphi(t,h(\mathbb{S}^2) )  ) \\
 			& \leq \int_{\mathbb{S}^2} \sqrt{ 2 \sum_{J \subset{1,\ldots,n}: |J|=2 } M^2 e^{ -2 \lambda t} \cdot \textrm{det} \left ( \frac{ \partial h_J}{\partial \theta } \right )^2 } d \theta_1 d \theta_2 \\ &= \sqrt{2} M e^{- \lambda t} \mu_2( h(\mathbb{S}^2)). \end{align*}
 		The latter inequality however contradicts (\ref{tobecontradicted}) for all $t$ sufficiently large.
 		An analogous proof applies to any invariant set which is the image of an injection $h$ of $\mathbb{S}^k$ for $k > 2$, thanks to Corollary \ref{usefullemma}. Notice that convexity of $K$ was not
 		crucial so far in the proof.   
 		
 		We consider next the case $k=1$. Let $h: \mathbb{S} \rightarrow K$ be a class $\mathcal{C}^1$ map such that $h( \mathbb{S})$ is invariant. Pick any point $\tilde{x} \in K$.
 		We consider the map $\tilde{h}: \mathbb{D} \rightarrow K$ 
 		defined as $ \tilde{h} (z):= (1-|z|) \tilde{x} + |z| h(z/|z|)$ (this is a convex combination of $\tilde{x}$ and points of $h( \mathbb{S})$ and it therefore belongs to $K$ by convexity of the set). By construction $\tilde{h}$ defines a surface (not necessarily smooth everywhere) such that $\tilde{h} ( \partial \mathbb{D}) = h( \mathbb{S})$.
 		Notice that, by invariance of $h ( \mathbb{S})$ this is also true of the map
 		$\varphi(t, \tilde{h}( \cdot ) )$, i.e., $\varphi(t,\tilde{h}( \partial \mathbb{D} )) = h ( \mathbb{S} )$.
 		Our goal is to estimate the area of $\varphi(t,\tilde{h} ( \mathbb{D}))$.
 		This can be computed according to:
 		\begin{align}
 			\label{areaforbound2}
 			& \mu_2 ( \varphi(t,\tilde{h}(\mathbb{D})) ) \\ \nonumber & =   {\Huge  \int_{\mathbb{D} } }   \sqrt{ \sum_{I \subset \{1,\ldots, n\}:|I|=2 }  \left [\textrm{det} \left ( \frac{ \partial}{\partial z} \varphi_I (t, \tilde{h}(z))  \right ) \right ]^2  } \, d z_1 d z_2.  
 		\end{align}	
 		Following the same steps as in the previous proof we see that:
 		\begin{equation}
 			\label{expdiscdecay}
 			\mu_2 ( \varphi(t,\tilde{h}(\mathbb{D})) ) \leq \sqrt{2} M e^{- \lambda t} \mu_2 ( \tilde{h} ( \mathbb{D}  ) ). 
 		\end{equation}   
 		However, by \cite{douglas}, there exists a surface of minimal area which is bounded by a given contour. This surface may, in general, present self-intersections depending on how complex is the contour (for instance due to the presence of knots).
 		Moreover, the surface of minimal area has positive measure $\underline{\mu}>0$, \cite{douglas, muldowney}.
 		This, however contradicts (\ref{expdiscdecay}) for all sufficiently large $t>0$.
 		This concludes the proof of the Theorem.
 	\end{proof}
 	\begin{remark}
 		We point out that replacing second additive compound matrices in (\ref{2var}) with
 		the standard Jacobians, that is the case of $k=1$ instead of $k=2$, yields classical variational criteria for exponential incremental stability. Theorem \ref{muld} hence relaxes such assumptions since, by virtue of Corollary \ref{usefullemma}, exponential convergence for $k=1$ (as needed in incremental stability) implies exponential convergence for all higher values of $k$. The converse is obviously not true. \\
 		 {It is worth pointing out that condition (\ref{expstability}), in combination with the other assumptions of Theorem \ref{muld}, is only a sufficient condition for ruling out oscillatory behaviors. In the case of constant matrices the second additive compound is asymptotically stable iff the real part of the sum of the dominant and subdominant eigenvalues is negative.
 			This affords existence of an unstable eigenvalue, provided the subdominant one is sufficiently within the left-hand side of the complex plane. Likewise, in a time-varying context, one can expect exponential stability as in (\ref{expstability}) provided the dominant and subdominant Lyapunov exponents have negative sum.}
 	\end{remark}
 	
 	\begin{remark}
 		Notice that our conditions are also independent of so called Dual Lyapunov functions, as introduced by Rantzer, \cite{rantzer}. Specifically Rantzer makes use of the derivative of $n$-forms along the flow in order to impose an expansion condition on the volume everywhere away from the equilibrium of interest.
 		This implies almost global convergence to the equilibrium under suitable integrability conditions on the considered density functions.
 	\end{remark}
 	
 	\section{Robust Lyapunov criterion for persistently-excited differential inclusions}
 	In our subsequent treatment of BINs, we are interested in studying notions of robust non-oscillation. We interpret ``robustness'' in the control theory sense of structured uncertainties. Hence, we study  a class of uncertain dynamical systems.  Given the dynamical system \eqref{2var}, we want to study the case in which the dynamics of $\delta^{(2)}(t)$ can be embedded in a Linear Differential Inclusion (LDI).  For simplicity, we denote $z(t):=\delta^{(2)}(t)$.
 	
 	\subsection{Common Lyapunov functions for LDIs}
 	Similar to our previous works \cite{MA_cdc14,plos}, we seek to find a convex PWL Lyapunov function  {$V: \mathbb{R}^N \rightarrow \mathbb{R}_{\ge 0}$} of the following form:
 	
 	\begin{equation}
 		\label{convexlinearlyap}
 		V(z)= \max_{k \in \{1, \ldots, L \}} c_k^T z,
 	\end{equation}
 	for some vectors $c_1, \ldots, c_L \in \mathbb{R}^N$, where $N:={n \choose 2 }$  {to be evaluated for} $z(t)=\delta^{(2)}(t)$ as defined in \eqref{kvar}.
 	
 	We state the following definition :
 	\begin{definition}\label{commonLF}
 		Let the matrices $A_1,..,A_s \in \mathbb R^{N \times N}$, and a locally Lipschitz function $V:\mathbb R^N \to \mathbb R_{\ge 0}$ be given. For each $\varepsilon \geq 0$ let $\mathcal{A}_{\varepsilon}$ denote the set:
 		\begin{equation}\label{A_e} \mathcal{A}_{\varepsilon}= \left \{ \sum_{\ell=1}^s \alpha_\ell A_{\ell}: \alpha_{\ell} \geq \varepsilon, \ell=1 \ldots s \right \}. \end{equation}
 		Then,	we say that $V$ is a common non-strict Lyapunov function for the LDI
 		\begin{equation}\label{peldi}\dot z(t) \in \mathcal{A}_{\varepsilon} z(t)\end{equation}
 		if $V(z)$ is positive definite, (that is $V(z)>0$ for all $z \neq 0$) and it satisfies $\nabla V(z) A z \le 0$ whenever $\nabla V(z)$ exists and for all $A \in \mathcal{A}_{\varepsilon}$.
 	\end{definition}
 	\begin{remark} We show in Lemma \ref{lem.lip} in the Appendix that the conditions given in Definition \ref{commonLF} are necessary and sufficient for the time-derivative of $V$ (defined as the upper Dini's derivative) to be non-positive when evaluated over an arbitrary trajectory of the LDI. The details are given in the Appendix.
 	\end{remark}
 	
 	The following characterization is standard,  {but we include a proof in the Appendix} to make the discussion self-contained.
 	\begin{lemma}\label{lem2} Let the matrices $A_1,..,A_s$, and a locally Lipschitz function $V:\mathbb R^N \to \mathbb R_{\ge 0}$ be given. Then,	$V$ is a common Lyapunov function for the LDI (\ref{peldi}) iff $V$ is positive definite and
 		\begin{equation}
 			\label{commonlf}
 			V( e^{{A} t} z ) \leq V(z), \qquad \forall \, z, \forall \, t \geq 0, \; \forall \, A \in \mathcal{A}_{\varepsilon}.
 		\end{equation}
 	\end{lemma}

 	\subsection{Asymptotic stability and LaSalle's argument}
 	Notice that the  individual subsystems only need to fulfill the non-strict inequality (\ref{commonlf}) which implies Lyapunov stability, and not asymptotic stability.  In order to prove uniform exponential stability of a differential inclusion on the basis of existence of a non-strict Lyapunov function, we need a LaSalle-like criterion in conjunction with some notion of persistence of excitation.  For this purpose, we will prove asymptotic stability of the differential inclusion \eqref{peldi} for every $\varepsilon>0$. We refer to \eqref{peldi} as a Persistently-Excited LDI (PELDI). Notice that,
 	\[ \mathcal{A}_{\varepsilon}  \subsetneq \mathcal{A}_0 = \textrm{cone} \left \{ A_1,A_2, \ldots, A_s \right \},\]
 	where ``cone'' denotes conic hull.
 	
 	Intuitively speaking this system is persistently excited since every vertex of the nominal differential inclusion (achieved for $\varepsilon=0$) takes part (at least with some $\varepsilon$ contribution)
 	to the formation of the state derivative direction.
 	This arises naturally in the context of BINs since a topology-based criteria based on the absence of critical siphons is sufficient to prove non-extinction of all chemical species (a property known also as persistence \cite{persistence}) and this leads to a potentially tighter embedding
 	as in (\ref{peldi}).

 	Hence, for a given LDI (\ref{peldi}) and the associated PWL Lyapunov function $V(z)$, we define the matrices given below:
 	\begin{equation}
 		\label{midef}
 		M_i := [ A_1^T c_i, A_2^T c_i, \ldots, A_s^T c_i],
 	\end{equation}
 	for all $i \in \{1,2 \ldots L\}$.

 	Our main result for this section is stated below.
 	\begin{theo}
 		\label{theolasalle}
 		Let $V(z)$ be a PWL common Lyapunov function for system (\ref{peldi}) with $\varepsilon=0$. Assume that
 		\begin{equation}	
 			\label{lasass}
 			\textrm{Ker} [M_i^T] =\{ 0 \}, \qquad \forall \, i \in \{1,2, \ldots, L \}. 
 		\end{equation}
 		Then, for all $\varepsilon>0$ the PELDI \eqref{peldi}
 		is uniformly exponentially stable.
 	\end{theo}

 	\begin{proof}
 		Fix any $\varepsilon>0$ and let $x(t)$ be an arbitrary solution of (\ref{peldi}). Since $V$ is a common Lyapunov function for (\ref{peldi}) with $\varepsilon=0$ it is a fortiori a common Lyapunov function for (\ref{peldi}) because of the inclusion $\mathcal{A}_{\varepsilon} \subset \mathcal{A}_0$. Hence,
 		$V(x(t)) \leq V(x(0))$ for all $t \geq 0$. Hence $x(t)$ is bounded (by positive definiteness and radial unboundedness of $V$). The Lyapunov function $V(x(t))$ is non-increasing along $x(t)$ and therefore it admits a limit as $t \rightarrow + \infty$.
 		Let $\bar{v} \geq 0$ be the value of this limit. The solution $x(t)$ approaches 
 		its non-empty $\omega$-limit set $\omega( x( \cdot) )$ and $V(\bar{x})= \bar{v}$ for all $\bar{x} \in \omega( x( \cdot))$. The set $\omega (x(\cdot))$ is weakly invariant, \cite{clarke}. We pick an arbitrary solution $\tilde{x} (t)$ of (\ref{peldi})
 		such that $\tilde{x}(t) \in \omega(x(\cdot))$ for all $t \geq 0$.
 		For each $t \geq 0$ we  {consider} the set of active vectors:
 		\begin{equation}
 			\mathcal{C}(t):= \{ k: V(\tilde{x}(t))=c_k^T \tilde{x}(t)   \}.
 		\end{equation}  
 		 {and define the corresponding set-valued map, $\mathcal{C} :t \mapsto 2^{\{1,\ldots,L\}}$.}
 		By continuity of $\tilde{x}(t)$ the set-valued map $\mathcal{C} $ is upper-semicontinuous, viz.
 		for any $t$ and any open neighborhood $U$ of $\mathcal{C}(t)$ there exists a neighborhood $N_{t}$ of $t$ such that $\mathcal{C} (N_{t}) \subset U$. Hence,
 		since $\mathcal{C}$ only takes discrete values, we see that the above inclusion can be strengthened to $\mathcal{C} (N_{t}) \subset \mathcal{C}(t)$.
 		Letting $t$ be a point where the cardinality of $\mathcal{C}(t)$ is minimal (which exists by finiteness of the set $\{1,2,\ldots, L \}$) we see that
 		$\mathcal{C} ( N_t ) = \mathcal{C} (t)$ and therefore there exists an interval $[t,\tilde{t}]$ ($\tilde{t}>t$) 
 		where %
 		$\mathcal{C}(\tau) = \mathcal{C}(t)$  for all $\tau \in [t, \tilde{t} ]$.
 		Pick any $k \in \mathcal{C}(t)$. We know that $V(\tilde{x}(\tau))
 		= c_k^T \tilde{x} (\tau)= \bar{v}$ for all $\tau$ in the considered interval. 
 		Hence, by definition of solution of (\ref{peldi}):
 		\[ c_k^T \dot{\tilde{x}} ( \tau ) = c_k^T \sum_{\ell=1}^s \alpha_\ell (\tau) A_\ell \tilde{x}(\tau) = \sum_{\ell=1}^s \alpha_\ell (\tau) c_k^T A_\ell \tilde{x}(\tau) = 0 \]
 		for some $\alpha_\ell (\tau) \geq \varepsilon$, and almost all $\tau \in [t,\tilde{t}]$.
 		Recalling that $c_k^T A_\ell \tilde{x}(\tau) \leq 0$, and using continuity of $\tilde{x}(\tau)$ this in turn implies:
 		\[ c_k^T A_\ell \tilde{x} ( \tau ) = 0, \qquad \forall \, \ell \in \{1,2,\ldots, L \}\; \forall \tau \in [t, \tilde{t}]. \]
 		Hence, $\tilde{x} (\tau)$ belongs to $\textrm{Ker}[ M_{k}^T]$, and by assumption
 		(\ref{lasass})
 		$\tilde{x} (\tau)=0$. By strong invariance of the origin, this implies $\omega(x(\cdot)) = \{0\}$ and therefore (see e.g. Theorem 2 in \cite{biotechprogressangelisontag}) uniform exponential stability of (\ref{peldi}) for all $\varepsilon>0$ follows, by a standard relaxation argument. 
 		
 	\end{proof}	
 	 {\begin{remark}
 			Conditions (\ref{lasass}) are used to rule out, using a first order derivative test, existence of non-zero solutions of (\ref{peldi})  evolving on a level-set of $V$ for some time-interval.  
 			As such, they could be relaxed by formulating higher order differential tests. This, however, would increase significantly the complexity of their verification. Such relaxation was not needed in practical examples.
 	\end{remark}}
 	
 	 {\begin{remark}
 			It is shown in \cite{blanchini2} that a BIN admitting  a non-strict polyhedral Lyapunov function 
 			is asymptotically stable iff a robust non-singularity condition
 			(for strictly positive linear combinations)
 			holds on the matrices defining the embedding of the nonlinear differential equation. Theorem (\ref{theolasalle}) differs in several respects. \\
 			It is, in fact, a stability result for a linear differential inclusion, rather than for nonlinear dynamics which are embedded within a linear differential inclusion. 
 			Notice also that the matrices $M_i$ in (\ref{midef}) both involve the Lyapunov function vectors $c_i$ and the dynamics of the switched system.
 			As such, condition (\ref{lasass})  is not immediately related to a condition of robust non-singularity which, by definition, only involves the matrices of the switched system. We cannot rule out that, on a deeper level, condition (\ref{lasass}) might be related or even equivalent to a robust non-singularity test.
 			
 	\end{remark} }

 	\section{Robust non-oscillation of  BINs}
 	In this section we study non-oscillation of BINs as described in \S 2.
 	\subsection{Background on BINs}
 	We use the standard notation \cite{erdi,feinberg87,plos}. A BIN (also called a ``Chemical Reaction Network'') is a pair $(\mathcal S, \mathcal R)$ with a set of admissible kinetics $\mathscr K_{\mathcal S,\mathcal R }$ to be defined below.
 	
 	\paragraph*{Stoichiometry} The finite set of  species is denoted by $\mathcal{S} := \{ S_1, S_2, \ldots, S_{n_s} \}$, which combine and transform through a finite set of  reactions, $\mathcal{R} := \{ \R_1, \R_2, \ldots, \R_{n_r} \}$.
 	A non-negative linear integer combination of species is  called a $\emph{complex}$, and an ordered pair of complexes  define a  reaction which is written customarily as:
 	\[ \R_j: \sum_{i=1}^{n_s} \alpha_{ij} S_i  \rightarrow \sum_{i=1}^{n_s} \beta_{ij} S_i,  \]	
 	with integer coefficients $\alpha_{ij}, \beta_{ij}$ (called the stoichiometry coefficients). These are usually arranged in a matrix $[\Gamma]_{ij}:= \beta_{ij}-\alpha_{ij}$, called the \emph{stoichiometry matrix}, whose $(i,j)$-entry specifies the net amount of molecules of $S_i$ produced or consumed by reaction $\R_j$. If $  \sum_{i=1}^{n_s} \beta_{ij} S_i \rightarrow \sum_{i=1}^{n_s} \alpha_{ij} S_i $ is also a reaction, then we say that $\R_j$ is \emph{reversible} and we write $\sum_{i=1}^{n_s} \alpha_{ij} S_i  \rightleftharpoons \sum_{i=1}^{n_s} \beta_{ij} S_i$. 
 	
 	\paragraph*{Kinetics}The kinetics of the BIN can be defined by introducing 
 	a non-negative state vector $x = [x_1,x_2,\ldots,x_{n_s}]^T$ quantifying the concentration of each species and a choice of kinetics, i.e., a functional expression for the rates at which the corresponding reaction takes place:
 	$
 	 {R (\cdot)}: \mathbb R_{\ge 0}^{n_s} \rightarrow \mathbb R_{\ge 0}^{n_r}. 
 	$
 	The function  {$R(\cdot)$} can take many forms, and we assume that it satisfies basic  {smoothness and} monotonicity requirements defined as follows:
 	\begin{enumerate}
 		\item[\bf A1.] $R_j(x)$ is continuously differentiable, $j=1,..,n_r$;
 		\item[\bf A2.] if $\alpha_{ij}>0$, then $x_i=0$ implies  $R_j(x)=0$; %
 		\item[\bf A3.] ${\partial R_j}/{\partial x_i}(x) \ge 0$ if $\alpha_{ij}>0$ and ${\partial R_j}/{\partial x_i}(x)\equiv 0$ if $\alpha_{ij}=0$;
 		\item[\bf A4.] The inequality in   {A3} holds strictly for all positive concentrations, i.e when $x \in \mathbb R_+^n$.
 	\end{enumerate}
 	
 	Condition A2 represents the fact that a reaction cannot occur when any of its reactants is missing. Conditions A3 and A4  require that, at least in the interior of the positive orthant, rates be strictly monotone functions of reactants' concentrations. 
 	Furthermore, A3  specifies that only reactants can influence the rate of any reaction.
 	If a reaction rate $R$ satisfies A1-4 we say that it is \emph{admissible}. The set of all admissible reaction rates of a given BIN $(\mathcal S,\mathcal R) $ is denoted by $\mathscr K_{\mathcal S,\mathcal R }$.
 	
 	A typical choice of kinetics are the so called \emph{Mass-Action} kinetics,
 	which correspond to the following polynomial expression:
 	\begin{equation}
 		\label{ma}
 		R_j(x) = k_j \prod_{i=1}^{n_s} x_i^{\alpha_{ij}}, 
 	\end{equation}
 	for some constant parameter $k_j>0$ and with the convention that $a^0=1$ for all $a \in \mathbb{R}$.    
 	
 	\paragraph*{Dynamics} the dynamical system associated to the BIN is by definition:
 	\begin{equation}
 		\label{crnsys}
 		\dot{x} = \Gamma R(x).
 	\end{equation}
 	This is a (generally) nonlinear,  positive system, meaning that solutions have non-negative coordinates given that the initial conditions do.  For each initial condition $x_0$,
 	the affine space $\mathcal C_{x_0}=x_0 + \textrm{Im} [\Gamma]$  is so that the corresponding solution
 	$\varphi(t,x_0)$ belongs  {to $\mathcal C_{x_0}$} for all $t \geq 0$, i.e., $\mathcal C_{x_0}$ is forward invariant. Hence, the system dimension is often reduced by taking into account an independent set of conservation laws (viz. vectors in $\textrm{Ker}( \Gamma^{T} )$) and regarding the flow induced by (\ref{crnsys}) 
 	as parametrized by the total amount of each conservation law, and evolving on a lower dimensional space defined by the corresponding stoichiometry class. 
 	This is the approach that we will pursue also throughout this paper. In particular,
 	we will choose a basis for $\textrm{Ker} ( \Gamma^T)$  {(assumed of dimension $c$) as } $\{v_1,v_2, \ldots, v_c \}$
 	and complete it to a basis of $\mathbb{R}^n$, $\{v_1,v_2, \ldots v_c, v_{c+1}, \ldots, v_n\}$ so that, defining the matrix:
 	\[  T = [v_1,v_2, \ldots, v_n ]^T    \]
 	we may define the system in $\tilde{x}$ coordinates according to $\tilde{x} =T x$.
 	Accordingly the new equations read:
 	\begin{equation}\label{odeT}\dot{\tilde{x}} =   T \Gamma R (T^{-1} \tilde{x} ) = \left [ \begin{array}{c} 0_{n_c} \\
 			\Gamma_r R(T^{-1} \tilde{x}) \end{array} \right ],
 	\end{equation}
 	where $\Gamma_r = [v_{c+1}, \ldots, v_{n}]^T \Gamma$ is a reduced stoichiometry matrix.
 	Of course, the natural state-space in $\tilde{x}$ coordinates, (i.e., $T \mathbb R_{\ge 0}^n$) is not necessarily the  positive orthant, but possibly a subset of it (as the individual vectors $v_i$, $i=1,\ldots, n$ are often chosen to be non-negative).
 	Notice that, the vector $\tilde{x}$ can be partitioned according to $[ \tilde{x}_c, 
 	\tilde{x}_{d}]$ where $\tilde{x}_c$ corresponds to the first $c$ components of $\tilde{x}$ (which are constant along solutions) while $\tilde{x}_d$ corresponds to the remaining $n-c$ coordinates evolving according to non-trivial dynamics.
 	
 	\paragraph*{Siphons}   Since \eqref{crnsys} evolves on the positive orthant, certain trajectories might approach the boundary of the orthant asymptotically, i.e., some species might go extinct. If no species becomes extinct for every positive initial state, then the dynamical system is said to be \emph{persistent}. In order to characterize persistence graphically, we need some definitions. Let $P\subset \mathcal S$ be a nonempty set of species. A reaction $\R_j \in \mathcal R$ is said to be an input reaction to $P$ if there exists $S_i \in P$ such that $\beta_{ij}>0$, while a reaction $\R_j \in \mathcal R$ is said to be an output reaction to $P$ if there exists $S_i \in P$ such that $\alpha_{ij}>0$. Then, the set $P$ is called a \emph{siphon} if each input reaction associated to $P$ is also an output reaction associated to $P$ \cite{persistence}. The species that correspond to the support of a non-negative conservation law automatically constitute a siphon. Hence, any siphon that contains the support of a conservation law is said to be \emph{trivial}. If a siphon is not trivial, then it is said to be \emph{critical}. If a BIN has no critical siphons, then \eqref{crnsys} is persistent for any choice of monotone kinetics \cite{persistence}.
 	
 	\paragraph*{Graphical representation} A BIN can be represented as a graph in various ways. We adopt the Petri-net representation \cite{petri08,persistence}, which is also equivalent to a species-reaction graph \cite{craciun06}. For a given BIN, species correspond to places, while reactions correspond to transitions. The incidence matrix of the Petri-net is simply the stoichiometry matrix $\Gamma$. An example will be discussed next.
 	
 	\paragraph*{Example} Referring to the motivational example \eqref{ptm}-\eqref{ptmRL}, the reactions are ordered as:
 	\begin{equation}\label{ptm_open}
 		L+ {Rc}   \xrightleftharpoons[\R_2]{\R_1} K, \ 
 		S+K \xrightleftharpoons[\R_4]{\R_3} C \lra^{\R_5} P+K, \ P \lra^{\R_6} S. 
 	\end{equation}
 	The concentrations $x_1,..,x_6$ correspond to the species $L, {Rc},K,S,C,P$, respectively. The ODE can be written as:
 	\begin{equation} \dot x = \left [ \begin{array}{rrrrrr} -1 & 1 &0 & 0 & 0& 0   \\ -1 & 1 & 0 & 0 & 0 & 0   \\ 1 & -1 & -1 & 1 & 0  & 0   \\ 0 & 0 & -1 & 1 & 1 & 1  \\ 0 & 0 & 1 & -1 & -1 & 0 \\ 0 & 0 & 0 & 0 & 1 & -1 \end{array} \right ]  \left [ \begin{array}{l} R_1(x_1,x_2) \\ R_2(x_3) \\ R_3(x_3,x_4) \\ R_4(x_5) \\ R_5(x_5) \\ R_6(x_6) \end{array} \right ] \label{ptmRL.ode}
 	\end{equation}
 	where the rates $R_j, j=1,..,6$ satisfy the Assumptions A1-4. Beyond these assumptions we don't assume that anything is known about them.  The Petri-net graph of the network is depicted in Figure \ref{f.ptm}-a). 
 	
 	The BIN \eqref{ptm_open} has   three conserved quantities which are the total receptor, the total ligand, and the total substrate. The corresponding conservation laws can be written as: $x_1+x_3+x_5=x_{1,tot}$, $x_2+x_3+x_5=x_{2,tot}$, $x_4+x_5+x_6=x_{3,tot}$.  Note that the network has no critical siphons and hence it is persistent. The conservation laws can be used to reduce the equation above from a six-dimensional to a three-dimensional ODE. For instance, we can choose the independent variables to be $x_1,x_3,x_6$ (corresponding to $L,K,P$). Hence $T$ can be written as:
 	\begin{equation}\label{e.T}T=\begin{bmatrix} 1 & 0 & 1 & 0 & 1 & 0 \\ 0 & 1 & 1 & 0  & 1 & 0 \\  0 & 0 & 0 & 1 & 1 & 1 \\ 1 & 0 & 0 & 0 & 0 & 0 \\ 0 & 0 & 1 & 0 & 0 & 0 \\  0 & 0 & 0 & 0 & 0 & 1    \end{bmatrix}. \end{equation}
 	For given  total positive conserved quantities $x_{1,tot},x_{2,tot},x_{4,tot} >0$, we obtain in this manner an ODE for the evolution of $\tilde x_d(t)=[x_1,x_3,x_6]^T(t)$, as follows:
 	\begin{equation}\label{red_example}\dot{\tilde x}_d= \left [ \begin{array}{rrrrrr} -1 & 1 &0 & 0 & 0& 0  \\ 1 & -1 & -1 & 1 & 0  & 0  \\ 0 & 0 & 0 & 0 & 1 & -1 \end{array} \right ] R({\tilde x}_d), \end{equation}
 	where 
 	\[ R({\tilde x}_d)= \left [ \begin{array}{l} R_1(x_1,x_1-x_{1,tot}+x_{2,tot}) \\ R_2(x_3) \\ R_3(x_3,x_{4,tot}-x_{1,tot}+x_1+x_3-x_6) \\ R_4(x_{1,tot}-x_1-x_3) \\ R_5(x_{1,tot}-x_1-x_3) \\ R_6(x_6) \end{array} \right ] .\]
 	 {While, as is well known, this change of coordinates conveniently achieves a dimensionality reduction of the underlying dynamics, it plays a crucial role in enabling the analysis of BINs by using second additive compound matrices. This is so because structural zero eigenvalues of the Jacobian are removed, opening up the possibility of establishing uniform exponential convergence of the associated variational equations.}   
 	\subsection{Lyapunov criteria for robust non-oscillation of BINs}
 	We apply the concept of non-oscillation to BINs. Since the reaction rates are not assumed to be known beyond satisfying assumptions A1-4, we aim at establishing a notion of \emph{robust non-oscillation}. 
 	
 	\begin{definition} Let a BIN $(\mathcal S, \mathcal R)$ be given. We say that it is {robustly non-oscillatory} if the associated dynamical system \eqref{crnsys} is non-oscillatory for every choice of kinetics $R \in \mathscr K_{\mathcal S,\mathcal R}$.
 	\end{definition}

 	We aim at proving the non-oscillatory nature of the %
 	dynamics by embedding the
 	variational equation associated to the second additive compound matrix within a linear differential inclusion. This is reminiscent of our approach for treating robust global stability for BINs \cite{MA_cdc14},\cite{plos}.  To this end, take the Jacobian of the $\tilde{x}$-dynamics \eqref{odeT}, as the principal submatrix of indices $\{c+1,\ldots n\}$:
 	\begin{equation}
 		\label{expressionJr}
 		J_r ( \tilde{x} )= \left [ T \Gamma \frac{ \partial R}{\partial x}  T^{-1} \right ]_{c+1,\ldots,n}= \Gamma_r 
 		\frac{ \partial R}{\partial x} T^{-1} \left [ \begin{array}{c} 0 \\ I_{n-c} \end{array} \right ]. 
 	\end{equation}
 	Accordingly, the variational equation associated to (\ref{crnsys}) can be rewritten
 	as:
 	\begin{equation}
 		\label{2varcrn}
 		\begin{array}{l}
 			\dot{\tilde{x}}_c = 0, \\
 			\dot{\tilde{x}}_n = \Gamma_r R ( T^{-1} \tilde{x} ), \\
 			{\dot\delta}^{(2)} = J_r ( \tilde{x} )^{(2)} \,  \delta^{(2)},
 		\end{array}
 	\end{equation}
 	which has the advantage of a smaller $\delta^{(2)}$ variable, of dimension $N:={ {n-c} \choose 2 }$, driven by a $(n-c)$-dimensional flow,
 	parametrized by the initial condition $\tilde{x}_c (0)$.  
 	As in classical embedding approaches, \cite{MA_cdc14},\cite{blanchini},\cite{plos}, one may write $J_r$ as a positive combination of rank-one stable matrices, where each matrix corresponds to a \emph{reaction-reactant pair}. The set of all such pairs is denoted as: \begin{equation}\label{e.P} \mathcal P=\{(j,i)|S_i~\mbox{participates in the reaction}~\R_j\}.\end{equation}
 	Let $s$ be the cardinality of $\mathcal P$. Then, 
 	\[ \frac{ \partial R}{\partial x} = \sum_{i,j} e_j e_i^T \frac{ \partial R_j}{\partial x_i}=\sum_{\ell=1}^s \rho_\ell(t) e_j e_i^T, \]
 	where $\rho_\ell :=   \partial R_{j_\ell}/\partial x_{i_\ell}, (j_\ell, i_\ell) \in \mathcal P, \ell=1,..,s$.
 	By substitution into (\ref{expressionJr}), we get:
 	\begin{equation} \label{ldi_embedding}
 		J_r = \sum_{\ell=1}^s \rho_\ell \left ((\Gamma_r e_j) ( e_i^T T^{-1}  \left [ \begin{array}{c} 0 \\ I_{n-c} \end{array} \right ]  )
 		\right )=: \sum_{\ell=1}^s \rho_\ell A_\ell. \end{equation}
 	Since the second additive compound is \emph{linear} in the entries of the original matrix, we get
 	\begin{equation}\label{exprembedding}
 		J_r^{(2)} = \sum_{\ell=1}^s \rho_\ell A_\ell^{(2)}. \end{equation}
 	Therefore, we study  (\ref{2varcrn})
 	by studying the LDI:
 	\begin{equation} 
 		\label{secondvarsw}
 		\dot{\delta^{(2)}} (t) \in \textrm{cone} \{ A_1^{(2)}, \ldots, A_s^{(2)} \} \delta^{(2)} (t) ,
 	\end{equation}
 	where $A_i$ are the corresponding rank one matrices as in  {(\ref{ldi_embedding})}.
 	The main result for this section is a theorem to guarantee uniform exponential stability of the $\delta^{(2)}$-subsystem in (\ref{2varcrn}) so that one may apply 
 	Theorem \ref{muld} with ease to BINs with uncertain kinetics.
 	\begin{theo}
 		\label{expcrn}
 		Let a BIN $(\mathcal S,\mathcal R)$ be given, and assume that it does not have critical siphons. 
 		Assume that the associated LDI (\ref{secondvarsw}) admits a 
 		PWL common Lyapunov function as in (\ref{convexlinearlyap})
 		fulfilling the additional conditions (\ref{lasass}). Then,  for any compact
 		$K \subset (0,+\infty)^n$ there exist $M, \lambda >0$, such that
 		for all $\tilde{x} (0) \in T K$ (i.e., the image of $K$ under the linear map $T$) and all $\delta^{(2)} (0) \in \mathbb{R}^{N}$ the corresponding solutions of (\ref{2varcrn}) fulfill
 		\[ | \delta^{(2)} (t) | \leq M e^{- \lambda t } |\delta^{(2)} (0)| \qquad \forall t \geq 0. \] 
 	\end{theo}
 	
 	The remainder of this section is dedicated to the proof of Theorem \ref{expcrn}. To that end,  we need to introduce some additional concepts and an improved version of the so called \emph{Siphon Lemma}, to be defined below.
 	For a compact set $K$, we denote the corresponding $\omega$-limit set as:
 	\begin{align*}
 		& \omega(K) \\ &= \{ x \in \mathbb{R}^n: \exists \, t_{n} \rightarrow + \infty, x_{n} \in K:\lim_{n \rightarrow + \infty} \varphi(t_n, x_n) = x \}.
 	\end{align*}
 	Notice that by construction this set contains $\bigcup_{x_0 \in K} \omega(x_0)$.
 	It is, however, a potentially bigger set.
 	For this reason the following is an improved version of the siphon Lemma, \cite{persistence,anderson}.
 	\begin{lemma}
 		\label{siphonrevisited}
 		Let $K \subset (0,+\infty)^n$ be compact and assume that $y \in \partial (0,+\infty)^n \cap \omega(K)$. Then, $\{S_i \in \mathcal{S}: y_i = 0 \}$ is a siphon. 	
 	\end{lemma}
 	We recall that the original siphon lemma only states this property for $K$ being a singleton.   We prove it in the Appendix for the case of compact sets {, thus generalizing} the proof presented in \cite{anderson}. This opens up the possibility of achieving structural criteria for uniform persistence in BINs.
 	In fact, (see \cite{hale}, pag. 8), the following holds for $\omega(K)$:
 	\begin{lemma}
 		\label{deep}
 		Consider a continuous flow and a compact set $K$, such that 
 		$\textrm{cl} \left ( \bigcup_{t \geq 0} \varphi(t,K)  \right )$ is bounded.
 		Then $\omega(K)$ is non-empty, compact, invariant and uniformly attracts $K$. 
 	\end{lemma}	
 	We are specifically interested in compactness of $\omega(K)$. This is crucial, since the property doesn't necessarily hold for $\bigcup_{x_0 \in K} \omega(x_0)$. 
 	Our main result hinges upon the following Lemma of independent interest.
 	\begin{lemma}
 		\label{eventuallyK}
 		Consider a chemical reaction network with uniformly bounded solution,  {i.e.,}
 		for all compact $K \subset \mathbb{R}_{\geq 0}^n$, there exists  $\tilde{K}$ compact such that $\varphi(t, K) \subset \tilde{K}$ for all $t \geq 0$.
 		Assume that all siphons are trivial.
 		Hence, 
 		for any compact $K \subset (0,+\infty)^n$, there exist $\varepsilon>0$ and
 		$\tilde{K}$ compact in $[\varepsilon,+ \infty)^n$ such that
 		$\varphi (t, K) \in \tilde{K}$ for all $t \geq 0$.
 	\end{lemma}
 	\begin{proof}
 		Let  $K \subset (0,+\infty)^n$ be arbitrary. By assumption $\varphi(t,K)$ is uniformly bounded, hence by Lemma \ref{deep}, $\omega(K)$ is non-empty and compact.
 		Its intersection with $\partial [0,+\infty)^n$, on the other hand, is empty, since
 		any point $y \in \omega(K) \cap   	\partial [0,+\infty)^n$  
 		fulfills that $\{ S_i: y_i =0 \}$ is a siphon,  by virtue of Lemma \ref{siphonrevisited}.
 		By the triviality of siphons, in turn, this amounts to existence of a non-negative conservation law $v \neq 0$ such that $v^T y =0$. This contradicts definition of $y$ since,
 		$y = \lim_{n \rightarrow + \infty} \varphi (t_n, \xi_n)$ for $\xi_n \in K$ and
 		as a consequence:
 		\begin{align*} 0 = v^T y =  \lim_{n \rightarrow  \infty} v^T \varphi (t_n, \xi_n)  {= \lim_{n \rightarrow  \infty} } v^T \xi_n 
 			\geq \min_{\xi \in K} v^T \xi > 0. \end{align*}
 		As a consequence, $\varepsilon:=\min_{\xi \in \omega(K) } \min_i \xi_i >0$.
 		We see that $\omega(K) \subset [\varepsilon, +\infty)^n \cap \tilde{K}$, where 
 		$\tilde{K}$ is as in the statement of the Lemma.
 		Moreover, $\omega(K)$ uniformly attracts $K$, so that
 		there exists $T>0$ such that for all $t \geq T$,
 		$\varphi(t, K) \subset [\varepsilon/2, + \infty) \cap \tilde{K}$. 
 		Finally, combining this latter inclusion, with the fact that solutions
 		$\varphi(t,K)$ are uniformly away from the boundary over any compact interval, i.e for $t \in [0,T]$ we prove the claim.	
 	\end{proof}
 	We are now ready to prove Theorem \ref{expcrn}.
 	\begin{proof}
 		Let $K \subset (0,+\infty)^n$ be an arbitrary compact.
 		By Lemma \ref{eventuallyK}, there exist {s} $\varepsilon>0$ and
 		$\tilde{K}$ compact in $[\varepsilon,+ \infty)^n$ such that
 		$\varphi (t, K) \in \tilde{K}$ for all $t \geq 0$.
 		Hence, by the strict positivity assumption on $\frac{ \partial R}{\partial x}$
 		there exist $\varepsilon>0$, such that the $\delta^{(2)}$ component of the solutions of (\ref{2varcrn}) can be embedded in that of a PELDI as in (\ref{peldi}).
 		The Theorem follows thanks to the fulfillment of conditions (\ref{lasass}) and by virtue of Theorem \ref{theolasalle}. 
 	\end{proof}

 	\section{Construction and existence of PWL Lyapunov functions}
 	In this section, we provide a fast iterative method for constructing Lyapunov functions, and also interpret the LDI in discrete-time settings.
 	\subsection{A fast iterative construction algorithm}
 	 {Construction of PWL Lyapunov functions is a longstanding problem in systems and control \cite{brayton79,michel84}, and several iterative algorithms have been proposed \cite{blanchini96}. Along similar lines,} we have proposed an iterative algorithm for constructing PWL Lyapunov functions in 
 	our previous works \cite{MA_cdc13,MA_TAC,plos}  where the dynamics can be embedded in a rank-one LDI. Since the second compound matrices \eqref{secondvarsw} are of rank $N-1$, we will generalize the aforementioned approach to handle such cases. 
 	
 	The PWL function \eqref{convexlinearlyap} satisfies the non-increasingness condition in Definition \ref{commonLF} if we have $\nabla V(z) A_\ell^{(2)} z \le 0$ whenever $\nabla V(z)$ exists. Note that we can write the following \[\nabla V(z)=c_k^T~\mbox{for all}~z \in \left\{ z\left | c_k^T z = \max_{j\in\{1,..,L\} }c_j ^Tz\right .\right \}^\circ,\]
 	where ``$\circ$'' denotes interior. Therefore, we need the following condition to be satisfied  $\forall \ell=1,..,s, \forall k=1,..,L$
 	\begin{equation}\label{Vdot}c_k^T A_\ell^{(2)} z \le 0~\mbox{whenever}~c_k^T z = \max_{j\in  {\{} 1,..,L  {\}}} c_j^T z. \end{equation}
 	In other words, the time-derivative of the $k$th linear component $c_k^T z$ needs to be non-positive only when the $k$th linear component is \emph{active}. 
 	
 	Since we are looking for robust, i.e., kinetics-independent conditions, we need to to impose a geometric condition relating the vectors $c_1,..,c_L$ with the matrices $A_1^{(2)},...,A_s^{(2)}$.  This can be achieved by noting that the \eqref{Vdot} is automatically satisfied if $-c_k^T A_\ell^{(2)}$ lies in the conic span of $\{c_k^T-c_j^T | j=1,..,L, j\ne k\}$. By the Farkas Lemma \cite{rockafellar}, \eqref{Vdot} is satisfied if there exist scalars $\lambda_{j}^{(k\ell)}\ge 0, j=1,..,L$, with $ \sum_{j\ne\ell}\lambda_{j}^{(k\ell)} >0$ such that
 	\begin{equation} \label{conicspan} - c_k^T A_\ell^{(2)} = \sum_{j\ne k} \lambda_{j}^{(k\ell)} (c_k^T - c_j^T). \end{equation}
 	Hence verifying the non-increasingness of the RLF reduces to satisfying the condition \eqref{conicspan}. 
 	
 	The algorithm starts with an \emph{initial} matrix $C_0=[c_1,..,c_{L_0}]^T \in \mathbb R^{L_0\times {N  }} $, where $N:={ {n-c} \choose 2 }$, and we let $V_0(z)= \max_{k \in \{1,..,L_0\}} c_k^T z $. We  choose $C_0=\mbox{diag}\left[I_{N }, - I_{ N }\right]$ to guarantee positive-definiteness of $V$.
 	
 	For each $c_k$ (amongst the rows of $C_0$), and for each $\ell$, we need to verify that condition \eqref{conicspan} is satisfied.  If not, we compute a new row $c^*$ chosen as to satisfy $-c_k^T A_\ell^{(2)}= c_k-c^*$. Hence,
 	\begin{equation} \label{newvector} {c^*}^T:= c_k^T (A_{\ell}^{(2)}+I).\end{equation}
 	
 	The new vector is appended to the matrix $C_0$ to yield a new matrix $C_1:=[C_0^T,c^*]^T$. The same process is repeated for each row vector of the coefficient matrix until either no new vectors need to be added or that the number of iterations exceeds a predefined number.
 	
 	There can be many variations on the basic recipe above.  Hence, we  {state the following}:
 	\begin{theo}\label{th_algorithm}
 		Given a network $( {\mathcal{S},\mathcal{R}})$.   If  Algorithm \ref{algorithm}  terminates successfully, then  $V$  is a common Lyapunov for the LDI $\dot z \in \mbox{cone}\{A_1^{(2)},..,A_s^{(2)}\}$. 
 	\end{theo}
 	\begin{algorithm}[h]
 		\SetAlgoLined

 		{\textbf{Parameters:}  $M$ as the upper maximum number of iterations. }\\
 		
 		{\textbf{Initialization:} Set $\mbox{flag}=0$, $C_0=\mbox{diag}\left[I_L, - I_L\right]$, %
 			$k := 1$, $L:={{n-c} \choose 2}$.} \\

 		\While{$k<M$ and $flag=0$}{

 			\For{$\ell \in \{1,..,s\}$}{
 				\If{$c_k^T A_\ell^{(2)} \ne 0$}{
 					$c^*:= c_k^T (A_\ell^{(2)}+I)$ ; 		\\
 					\If{ $c^* \ne c_\ell$ for $\ell=1,..,k$  }{set $C:=[C^T, c^{*T}]^T $;} 
 				}
 			}
 			$k:=k+1$; \\
 			
 			$L:=$ number of rows of $C$; \\
 			
 			\If{ $L<k$  }{set flag:=1;} 
 		}		
 		\eIf{$\mathrm{flag}=1$}
 		{ Success. $V(z)=\max_{k=0,..,m} c_k^T z$ is the desired function}
 		{	The algorithm did not converge within the prescribed upper maximum number of iterations. }

 		\caption{Iterative construction of PWL RLFs.} \label{algorithm}
 	\end{algorithm}

 	\subsection{Existence of PWL Lyapunov functions for LDIs}
 	Recall that the dynamics of a BIN can be embedded in an LDI of rank-one matrices \eqref{ldi_embedding}, and that the dynamics of  (\ref{2varcrn}) can be studied by the LDI of the corresponding second-additive compounds.
 	Our aim in this section is to provide alternative characterization for the existence of PWL Lyapunov functions for LDIs of rank-one stable matrices and their second compounds. We will     use their specific properties   (and their matrix exponentials) in order to simplify the test of property (\ref{commonlf}). Some of the results in this subsection  recover the discrete-time approach first introduced in \cite{blanchini} for studying the stability of BINs.

 	We start by stating  the following result.
 	\begin{lemma}\label{lem.exmrank1}
 		 {For a square rank one matrix $A= v w^T$ (and non zero vectors $v,w \in \mathbb{R}^N$ for any integer $N>0$) } the following expression holds:
 		\begin{equation}
 			\label{exmrank1}
 			e^{At} = I + v w^T  \int_0^t e^{ ( w^T v) \tau} \, d \tau.
 		\end{equation}
 	\end{lemma}
 	Notice that the exponential inside the integral is a scalar exponential. Hence, a non-trivial rank one linear system (with $A \neq 0$) is globally stable if and only if $w^{T} v < 0$, (in fact for $w^T v >0$ exponential instability arises, while for 
 	$w^T v=0$ the matrix exponential grows linearly in time).
 	Hence, without loss of generality we limit our discussion to rank one switched linear systems such that $w_\ell^T v_\ell <0$ for all $\ell \in \{1,..,L\}$.
 	
 	We show that the matrix exponential of a rank-one stable matrix can be written always as a convex combination of $I$ and the asymptotic value of the matrix exponential. The same also holds for the matrix exponential of its second additive compound. This is stated in the following Lemma:
 	\begin{lemma}\label{convexcomb}
 		Let $A = v w^T$ be a stable $n \times n$ rank one real matrix, for suitable vectors $v, w \in \mathbb{R}^n$. Denote by $A^{(2)}$ the associated second additive compound matrix. Then,
 		\begin{enumerate}
 			\item $e^{At} = e^{(w^T v) t} I + (1- e^{(w^T v) t}) \Pi$, where  $\Pi= \lim_{t \rightarrow + \infty}  e^{A t }$.
 			\item $ e^{A^{(2)} t } = e^{(w^T v) t }  I + (1-e^{ (w^T v) t}) \Pi_2$, 	where $ \Pi_2 = \lim_{t \rightarrow + \infty}  e^{A^{(2)} t }.$
 		\end{enumerate}
 	\end{lemma}
 	\begin{proof} \strut  
 		\begin{enumerate}
 			\item Using Lemma \ref{lem.exmrank1}, we can write:
 			\[ e^{At} = I + \frac{ v w^T }{w^T v} ( e^{(w^Tv) t } - 1 ), \]
 			which can be rearranged into $e^{At} = e^{(w^T v) t} I + (1- e^{(w^T v) t}) \Pi$, where
 			\[ \Pi := \lim_{t \rightarrow + \infty } e^{At} = I - \frac{v w^T}{w^T v}. \]
 			\item  {Let $\mathbb{SK}_n$ denote the class of $n \times n$
 				real skew-symmetric matrices, viz. $\mathbb{SK}_n = \{ X \in \mathbb{R}^{n \times n}: X= - X^T \}$.
 				For any matrix $A \in \mathbb{R}^{n \times n}$, the linear operator $L$ defined as:
 				\[ L(X):= AX + X A^T \]
 				is an endomorphism in $\mathbb{SK}_n$, viz. $L:\mathbb{SK}_n \rightarrow \mathbb{SK}_n$. Moreover, the second additive compound matrix $A^{(2)}$ can be interpreted as a  representation of $L$, with respect to the canonical basis $\mathbb{B}_n:=\{ e_i e_j^T - e_j e_i^T,  i < j  \}$ of $\mathbb{SK}_n$, where $i$ and $j$ take values in $\{1,2, \ldots , n\}$, $e_i$ denotes the $i$-th element of the canonical basis of $\mathbb{R}^n$,
 				and elements of $\mathbb{B}_n$ are listed according to lexicographic ordering of the underlying index pairs
 				$\{i < j \}$,
 				see \cite{pantea14}.}
 			Hence, the matrix exponential $e^{A^{(2)} t }$ can equivalently be computed by looking at the operator induced by the solution of the linear matrix differential equation:
 			\[ \dot{X} = L(X).\]
 			This is well-known to be $X(t) = e^{A t} X(0) e^{A^Tt}$ which in the case of  {$A$ being of rank one} (assuming without loss of generality $w^Tv=-1$  {)} :
 			\begin{align*}&X(t)  =  (I - (e^{-t} -1)  v w^T ) X(0)  (I - (e^{-t} -1)  v w^T )^T \\
 				&  \!\!= X(0) - (e^{-t} - 1) v w^{T}  X(0) - (e^{-t}-1) X(0) w v^T \\ &   \qquad +
 				(e^{-t}-1)^2 v \underbrace{w^T X(0) w}_{=0} v^T      \\
 				& \!\!= e^{-t} X(0) + (1- e^{-t}) [ X(0) + v w^T X(0) + X(0) w v^T].   \end{align*}
 			Hence the result follows by noticing that
 			\[ \lim_{t \rightarrow + \infty }  X(t) = [ X(0) + v w^T X(0) + X(0) w v^T]. \]	
 			and letting $\Pi_2$ be the matrix associated to the operator
 			$L_{\infty} (X) := [ X + v w^T X + X w v^T]$ acting on real skew-symmetric matrices of dimension $n$.
 		\end{enumerate}
 	\vspace{-0.2in}
 	\end{proof}
 	This allows to recast condition (\ref{commonlf}) in a simpler way that does not directly involves time.
 	\begin{lemma}
 		\label{dtcheck}
 		Let the matrices $A_1,..,A_s \in \mathbb R^{N \times N}$, and a convex locally Lipschitz function $V:\mathbb R^N \to \mathbb R_{\ge 0}$ be given. Assume that $A_{\ell}, \ell=1,..,s$ are either stable rank-one matrices or their second additive compounds.  Then,	$V$ is a common Lyapunov function for the LDI $\dot z(t) \in \mbox{cone}\{A_1,..,A_s\}$ iff $V$ is positive definite and
 		\begin{equation}
 			\label{discretetimeV}
 			V( \Pi_\ell z ) \leq V(z), \qquad \forall \, z,  \forall \, \ell \in \{1,..,s\},
 		\end{equation}
 		where $\Pi_\ell:= \lim_{t\to \infty} e^{A_\ell t}$.
 	\end{lemma} 
 	\begin{proof}
 		Fix $\ell$.	Using Lemma \ref{convexcomb}, for any $t \geq 0$ let $\alpha \in [0,1]$ be such that $e^{A_\ell t} = \alpha I + (1- \alpha) \Pi$.
 		Assume that (\ref{discretetimeV}) holds. Then
 		\[ V( e^{A_\ell t} z )= V ( \alpha z+ (1-\alpha) \Pi_\ell  z) \leq
 		\alpha V(z) + (1-\alpha) V ( \Pi_\ell  z) \]
 		\[ \qquad \leq \alpha V(z) + (1- \alpha) V(z) = V(z). \]
 		Hence, condition (\ref{commonlf}) follows.
 		Conversely, let condition (\ref{commonlf}) hold. By letting $t$ go to infinity in both sides of the inequality and exploiting continuity of $V(x)$ we get:
 		\[ V( \Pi_\ell  z) = V \left ( \lim_{t \rightarrow + \infty} e^{A_\ell t} z \right )=\lim_{t \rightarrow + \infty} V (e^{A_\ell t}z ) \leq V(z). \]
 	\end{proof}
 	Lemma \ref{dtcheck} shows that common Lyapunov functions for continuous time rank-one linear systems (or their second-additive compounds) can in fact be tested by using the conditions typical of discrete time LDIs, in particular adopting in place of each matrix exponential $e^{A_\ell t}$ the corresponding projection matrix $\Pi_\ell $.     \label{s.discrete}
 	This has some advantages, in particular  as we may show the instability of a given LDI as we will demonstrate in the examples section. Further, we may consider a closed-form expression for $V(z)$ of the following form:
 	\begin{equation}
 		\label{dtcommonlyap}
 		V (z) := \sup_{L \in \mathbb{N}, w \in \{1,..,s\}^L}  \left | \left ( \prod_{k=1}^L \Pi_{w_k} \right ) z  \right |_{1, \infty}  
 	\end{equation}
 	where, for simplicity, either $1$ or $\infty$ norms (both piecewise linear) are adopted. 
 	 {For any initial condition $z$, the expression in (\ref{dtcommonlyap}) amounts to computation of the maximum $1$ or $\infty$ norm of all possible forward solutions of the discrete differential inclusion induced by $\Pi_{\ell}$, for $\ell=1,2,\ldots,s$.
 		For this reason,} $V(z)$ as defined above is well-posed (bounded) if and only if the corresponding LDI is stable. 
 	Notice that the supremum in equation (\ref{dtcommonlyap})  is taken over an infinite number of possible product combinations. In practice, it is often the case that only a finite number of such products actively contribute to the value of $V(z)$ over $\mathbb{R}^n$ and, as a consequence, a finitely verifiable construction algorithm for polytopic Lyapunov functions can be derived by using the above formula whenever it is realized that only words
 	of up to a fixed length actively contribute to the value of $V(z)$. 
 	
 	It can be noted that this alternative algorithm is computationally slower than Algorithm \ref{algorithm}, and it has yielded the same results that we got using Algorithm \ref{algorithm}. On the other hand, the second algorithm can be terminated quickly if the spectral radius of one of the products  in \eqref{dtcommonlyap} exceeds $1$ since this means that the corresponding LDI is exponentially unstable.
 	
 	 {
 		\begin{remark}
 			Alternative methods can be proposed for deriving the Lyapunov functions. This includes studying the corresponding LDI in reaction coordinates \cite{MA_cdc14,plos}, or via  the concept of duality. In particular, one may consider the LDI associated to  $(A_i^{(2)})^T = (A_i^T)^{(2)}$. Such LDI enjoys the same stability properties of the original one and any Lyapunov function for the latter can be transformed to a Lyapunov function for the first one using well-known techniques, see for instance \cite{blanchini}. 
 		\end{remark}
 	}

 	\section{Biochemical examples} 
 	\subsection{A PTM cycle regulated by the binding of a receptor and a ligand}
 	We continue studying the regulated PTM \eqref{ptm_open} which was first introduced in \cite{dissertation}. Its Petri-net is depicted in Figure \ref{f.ptm}-a). The ODE describing the network is given in \eqref{ptmRL.ode}.
 	This network is known to fulfill all necessary conditions for existence of a PWL RLF (either in species or rates coordinates) but whose global asymptotic stability is still an open problem \cite{plos}.

 	The reduced Jacobian \eqref{ldi_embedding} (defined via the transformation matrix \eqref{e.T}) is a linear (positive) combination of the following rank one matrices, 
 	\[\scriptsize A_1\!=\!\!\left[\begin{array}{rrr} -1 & 0 & 0\\ 1 & 0 & 0\\ 0 & 0 & 0 \end{array}\right]\!,
 	A_2\!=\!\!\left[\begin{array}{rrr} -1 & 0 & 0\\ 1 & 0 & 0\\ 0 & 0 & 0 \end{array}\right]\!,
 	A_3\!=\!\left[\begin{array}{rrr} 0 & 1 & 0\\ 0 & -1 & 0\\ 0 & 0 & 0 \end{array}\right]\!, \]
 	\[ \scriptsize A_4\!\!=\!\!\left[\begin{array}{rrr} 0 & 0 & 0\\ 0 & -1 & 0\\ 0 & 0 & 0 \end{array}\right]\!\!,
 	A_5\!\!=\!\!\left[\begin{array}{rrr} 0 & 0 & 0\\ -1 & -1 & 1\\ 0 & 0 & 0 \end{array}\right]\!\!,
 	A_6\!\!=\!\left[\begin{array}{rrr} 0 & 0 & 0\\ -1 & -1 & 0\\ 0 & 0 & 0 \end{array}\right]\!\!, \] \[\scriptsize
 	A_7=\left[\begin{array}{rrr} 0 & 0 & 0\\ -1 & -1 & 0\\ -1 & -1 & 0 \end{array}\right],
 	A_8=\left[\begin{array}{rrr} 0 & 0 & 0\\ 0 & 0 & 0\\ 0 & 0 & -1 \end{array}\right].\]
 	Notice that the LDI:
 	\[ \dot{x} (t) \in \{A_1,...,A_8\} x(t), \]
 	is not Lyapunov stable as there exists a combination of matrices exhibiting linear instability. In particular,
 	\[ e^{(A_4+A_7) t } = \left[\begin{array}{ccc} 1 & 0 & 0\\ \frac{{\mathrm{e}}^{-2t}}{2}-\frac{1}{2} & {\mathrm{e}}^{-2t} & 0\\ \frac{{\mathrm{e}}^{-2t}}{4}-\frac{t}{2}-\frac{1}{4} & \frac{{\mathrm{e}}^{-2t}}{2}-\frac{1}{2} & 1 \end{array}\right].
 	\]
 	For this reason we introduce the corresponding second additive compound matrices listed below:
 	\[\scriptsize A_1^{(2)}\!\!=\!\!\begin{bmatrix} -1 & 0 & 0\\ 0 & -1 & 0\\ 0 & 1 & 0 \end{bmatrix}\!\! ,
 	A_2^{(2)}\!\!=\!\!  \begin{bmatrix} -1 & 0 & 0\\ 0 & -1 & 0\\ 0 & 1 & 0 \end{bmatrix}\!\! ,
 	A_3^{(2)}\!\!=\!\!  \begin{bmatrix} -1 & 0 & 0\\ 0 & 0 & 1\\ 0 & 0 & -1 \end{bmatrix}\!\! ,\] \[ 
 	\scriptsize A_4^{(2)}\!\!=\!\!    \begin{bmatrix} -1 & 0 & 0\\ 0 & 0 & 0\\ 0 & 0 & -1 \end{bmatrix}\!\!, 
 	A_5^{(2)}\!\!=\!\!    \begin{bmatrix} -1 & 1 & 0\\ 0 & 0 & 0\\ 0 & -1 & -1 \end{bmatrix}\!\!,
 	A_6^{(2)}\!\!=\!\!    \begin{bmatrix} -1 & 0 & 0\\ 0 & 0 & 0\\ 0 & -1 & -1 \end{bmatrix}\!\!,\]\[ \scriptsize
 	A_7^{(2)}\!\!=\!\!    \begin{bmatrix} -1 & 0 & 0\\ -1 & 0 & 0\\ 1 & -1 & -1 \end{bmatrix}\!\!,
 	A_8^{(2)}\!\!= \!\!   \begin{bmatrix} 0 & 0 & 0\\ 0 & -1 & 0\\ 0 & 0 & -1 \end{bmatrix}\!\!,\]
 	and, rather than assessing global asymptotic stability we look at the slightly weaker notion of globally non-oscillatory behavior.
 	Hence, we study stability of the differential inclusion:
 	\begin{equation}
 		\label{secondswitch}
 		\dot{\delta^{(2)}}(t) \in \textrm{cone} \{ A^{(2)}_1, A^{(2)}_2, \ldots,
 		A^{(2)}_8  \} \delta^{(2)} (t). 
 	\end{equation} 
 	where $\delta^{(2)} (t)$ is a vector of dimension ${n \choose 2}$.
 	Application of Algorithm 1 results in the following suitable Lyapunov function for
 	system (\ref{secondswitch}):
 	\begin{equation} V( \delta^{(2)} ) = \max \{   |\delta^{(2)}_1|,|\delta^{(2)}_2|, |\delta^{(2)}_3|,| \delta^{(2)}_2+ \delta^{(2)}_3|, | \delta^{(2)}_2 - \delta^{(2)}_1 | \}.   \label{V_example}\end{equation}
 	Also, the  formula (\ref{dtcommonlyap}) (by adopting the $\infty$-norm) results in the same function.
 	
 	Moreover, modelling the network as a Petri Net (see Fig. \ref{f.ptm}-a))
 	one can show that it admits $3$ minimal siphons, $\{ R,K,C \}$, $\{ L,K,C\}$ and $\{ S,C,P \}$. These are trivial siphons, as they coincide with the support of a non-negative conservation law. 
 	Moreover conditions (\ref{lasass}) are fulfilled. Hence, the BIN is non-oscillatory by virtue of Theorems \ref{expcrn} and \ref{muld}, regardless of the specific choice of kinetics. 
 	
 	Additional analysis of the network is possible, The Jacobian is a $P_0$ matrix for any choice of kinetics, hence the network can not admit multiple non-degenerate steady states in a single stoichiometric class \cite{craciun05,banaji07}. In addition, it can be shown that the Jacobian is robustly non-degenerate in the interior of the orthant \cite{colaneri}, \cite{plos}. Furthermore, the boundary of any non-trivial stoichiometric class cannot contain any steady states due to the absence of critical siphons \cite{persistence}, hence no more than one steady state can exist in the interior of each stoichiometric class.  More recently, sum-of-square optimization has been used to show that the reduced Jacobian is Hurwitz at any steady state, i.e., each steady state is locally asymptotically stable relative to its stoichiometric class \cite{colaneri}.   The existence of at least one steady state follows by the Brouwer's fixed point theorem \cite{royden88} or Poincar\'{e}-Hopf theorem \cite{krasno84}. 
 	To summarize, each non-trivial stoichiometric class contains a unique locally asymptotically stable steady state and the network is robustly non-oscillatory.  Though global asymptotic stability is still technically open, this is a quite tight approximation.
 	
 	\begin{figure*}
 		\centering
 		\includegraphics[width=0.8\textwidth]{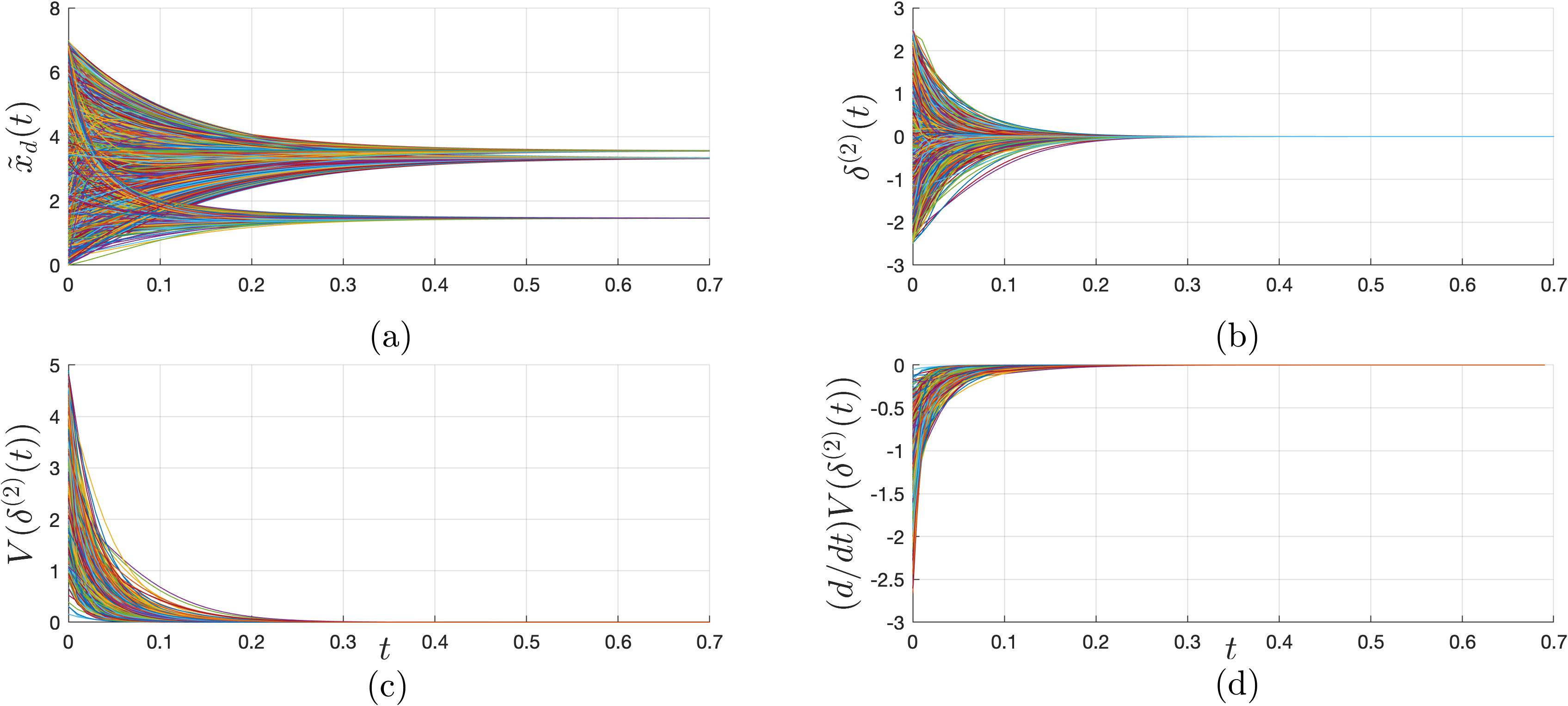}
 		\caption{ {\bfseries Sample trajectories of the regulated PTM with Mass-Action kinetics.}   {(a) Trajectories of \eqref{red_example} with 500 randomly selected initial conditions. Note that all trajectories converge to the unique steady state. (b) The corresponding trajectories of $\delta^{(2)}(t)$ (as defined in \eqref{2varcrn}) with   randomly chosen initial conditions $\delta^{(2)}(0)$. (c) The PWL Lyapunov function $V$ \eqref{V_example} evaluated over the trajectories of $\delta^{(2)}$ is decreasing as foretold by our results. (d) The time-derivatives of $V$ evaluated via MATLAB's command \texttt{diff} is negative for all $t\ge 0$. The chosen reaction rate vector (as in \eqref{ptmRL.ode}) is $R(x)=[5x_1 x_2,3x_3,5x_3x_4,x_5,2x_5,6x_6]$. The conserved quantities are $x_{1,tot}=x_{2,tot}=x_{4,tot}=15$. }} \label{simulation}
 	\end{figure*}
 	 {Figure \ref{simulation} shows sample trajectories of the system with Mass-Action kinetics and the corresponding PWL Lyapunov function \eqref{V_example} evaluated over the trajectories of $\delta^{(2)}$.}
 	
 	\subsection{A PTM cycle regulated by a kinase inhibitor}
 	In this subsection we discuss the network depicted in Figure \ref{f.ptm}-b). This network is interesting as we will show that the corresponding LDI is exponentially unstable.

 	The reactions are listed below:
 	\begin{equation}\label{ptm_I}
 		I+K   \xrightleftharpoons[\R_2]{\R_1} KI, \ 
 		S+K \xrightleftharpoons[\R_4]{\R_3} C \lra^{\R_5} P+K, \ P \lra^{\R_6} S. 
 	\end{equation}
 	The concentrations $x_1,..,x_6$ correspond to the species $I,KI,K,S,C,P$, respectively.

 	This network exhibits three conservation laws, $x_1 + x_2=\mbox{const}$, $x_2+x_3+x_5=\mbox{const}$ and $x_4 + x_5 + x_6=\mbox{const}$. Hence, each stoichiometry class is $3$-dimensional.
 	Choosing $x_2, x_5$ and $x_6$ as independent coordinates we achieve a reduced
 	Jacobian matrix of the following form:
 	\begin{align*} %
 		 \scriptsize  J_r= \left[\begin{array}{ccc} -\rho_{3,3}-\rho_{3,4}-\rho_{4,5}-\rho_{5,5} & -\rho_{3,3} & -\rho_{3,4}\\ -\rho_{1,3} & -\rho_{1,1}-\rho_{1,3}-\rho_{2,2} & 0\\ \rho_{5,5} & 0 & -\rho_{6,6} \end{array}\right],
 	\end{align*}
 	where $\rho_{j,i}:=\frac{\partial R_{j}}{\partial x_{i}},(j,i) \in \mathcal P$ are treated as arbitrary time-varying positive coefficients.
 	
 	The associated LDI, however, does not admit a common Lyapunov function. Indeed, by constructing products of the resulting $\Pi_\ell$ matrices (defined in \S \ref{s.discrete}), there exist finite products (of length $5$ or higher) with spectral radius strictly bigger than $1$.  
 	A Lyapunov function can instead be found for the embedding to the LDI of second additive compound matrices.
 	In particular,
 	\[ V( \delta^{(2)} ) = \max \{ | \delta_1^{(2)} |, | \delta_2^{(2)} |, | \delta_3^{(2)}|, | \delta_1^{(2)} - \delta_3^{(2)} |   \}. 
 	\]
 	is a suitable Lyapunov function. In addition, the Petri Net admits $3$ minimal siphons, $\{I,EI \}$, $\{ EI, K,C \}$, $\{S,C,P\}$, which are trivial. Again the main results of the paper can be applied to conclude that this is a robustly non-oscillatory dynamical system within each compact set included in the (strictly) positive orthant. Furthermore, similar to the previous example, it can be shown that each nontrivial stoichiometric class contains a unique positive steady state.

 	Similarly, the network in Figure \ref{f.ptm}-c) can be shown to be robustly non-oscillatory using Algorithm \ref{algorithm}.
 	
 	\section{Discussion}
 	We have proposed the notion of non-oscillation to be  studied as a useful verifiable property of nonlinear systems. A Lyapunov criteria has been proposed for robust non-oscillation. We have applied our theory to the study of BINs with general kinetics, and demonstrated the power of the theory for the study of regulated enzymatic cycles.

 	The failure of the existence a PWL RLF for the LDI associated to a BIN has no bearing on the actual properties of the BIN. While such conditions (existence of Lyapunov functions) are essentially necessary and sufficient for the study of stability in LDIs, they might  be conservative for the study of BINs.
 	These, in fact, are uncertain nonlinear systems merely embedded within an LDI but do not necessarily share all the dynamical behaviors of the LDI.
 	For instance, many BINs naturally have bounded solutions due to invariance of the positive orthant and existence of conservation laws, but this does not imply the resulting LDI will necessarily fulfil similar boundedness properties (invariance of the positive orthant is often not preserved in the embedding process).
 	
 	Although we have demonstrated the theory for systems which have unique steady states, the results are applicable to multistable systems, and finding a robustly non-oscillatory multi-stable BIN will be a highly interesting endeavour.
 	
 	 {To be concrete, and because of our interest in periodic or quasiperiodic behavior, we have restricted attention to parametrizations of invariant sets by tori, including circles.  However, the same method can be used to rule out invariant sets of positive measure that are parametrized by more general compact manifolds.}
 
 	\appendix
 	\subsection{Time-derivative of a locally Lipschitz Lyapunov function}
 	We include the following lemma and its proof. A similar lemma has been proven in \cite[Supplementary Information]{plos}.
 	\begin{lemma} \label{lem.lip}Let the matrices $A_1,..,A_s \in \mathbb R^{N \times N}$, a non-negative scalar $\varepsilon\ge 0$, and a locally Lipschitz function $V:\mathbb R^N \to \mathbb R_{\ge 0}$ be given. Let $\mathcal A_\varepsilon$ be as defined in \eqref{A_e}, and let $\dot z(t) \in \mathcal A_\varepsilon z(t)  $ be the corresponding LDI. Then for any trajectory $\varphi(t;z_0)$ of the LDI, we have:  $\frac{d}{dt} V(\varphi(t;z_0) )\le 0$ for all  $t \ge 0$, iff      $\nabla V(z) A_\ell z \le 0$ for all $z$ such that $\nabla V(z)$ exists and for all $\ell=1,..,s$.
 	\end{lemma}
 	\begin{proof} 
 		Fix $t$. Let $z:=\varphi(t;z_0)$ be a trajectory of the LDI, and let $\dot z:=\frac{d}{dt}\varphi(t;z_0) \in \mathcal A_\varepsilon z $.
 		We can write:
 		\begin{align}\nonumber\frac{d}{dt} V(z(t)) &= \limsup_{h \to 0^+} \frac { V(\varphi(t+h;z_0)) - V(\varphi(t;z_0))} h \\ \nonumber &= \limsup_{h \to 0^+} \frac{ V(\varphi(t;z_0)+h \tfrac{d}{dt}\varphi(t;z_0)) - V(\varphi(t;z_0))} h \\ & = \limsup_{h \to 0^+} \frac{ V(z+h \dot z) - V(z)} h. \label{Vdot2} \end{align}
 		For sufficiency, we just need to prove the following statement:  assume that $\nabla V(z) A_\ell z \le 0$ whenever $\nabla V(z)$ exists and for all $\ell=1,..,s$, then $ D_{\dot z} V(z) := \limsup_{h \to 0^+} ( V(z+h \dot z) - V(z))/ h \le 0$, for all $z \in \mathbb R^n$ and all $\dot z \in \mathcal A_\varepsilon z$.

 		Since $V$ is assumed to be locally Lipschitz, Rademacher's Theorem implies that it is differentiable (i.e., gradient $\nabla V(z)$ exists) almost everywhere \cite{clarke}. Recall that for a locally Lipschitz function the \emph{Clarke gradient} at $z$ is defined as  $\bar\partial V(z) :=\mbox{co} \partial V(z)$, where:
 		$\partial V(z):=  \{ p \in \mathbb R^n : \exists z_i \to z \, \mbox{with} \, \nabla V (z_i)$~exists, such that,~$ p^T=\lim_{i \to \infty} {\nabla V(z_i)}  \}. $
 		
 		Let $p \in \partial V(z)$ and $\dot z \in A_\varepsilon z$. Let $\{z_i\}_{i=1}^\infty$  be any sequence as in the definition of the Clarke gradient such that $\nabla V(z_i) \to p^T$. Furthermore, by the assumption stated in the Lemma, we have $\nabla V(z_i) A_\ell z_i \le 0$ for all $\ell$ and $i$. Since $\dot z =\sum_{\ell} \rho_\ell A_\ell z$ for some $\rho_1,..,\rho_s\ge \varepsilon$, then we can define  corresponding sequences $\{\rho_{1i}\}_{i=1}^\infty,..,\{\rho_{si}\}_{i=1}^\infty \subset [\varepsilon, \infty)$  such that $ \dot z_i:=\sum_\ell \rho_{\ell i} A_\ell z_i \to \dot z$. Hence, $\nabla V(z_i) \dot z_i \le 0$, $i\ge 1$. The definition of $p$ implies that $p^T \dot z \le 0$. Since $p$ is arbitrary, the inequality holds for all $p \in \partial V(z)$. \\
 		Now, let $p \in \bar \partial V(z)$ where $ p=\sum_{i} \lambda_i p_i$ is a convex combination of any $p_1,...,p_{n+1} \in \partial V(z)$. By the inequality above, $p^T \dot z = \sum_{i} \lambda_i (p_i^T \dot z ) \le 0. $ Hence, $p^T \dot z \le 0$ for all $p \in \bar \partial V(z)$. \\
 		As in \cite{clarke}, the Clarke derivative of $V$ at $z$ in the direction of $\dot z$ can be written as $D_{\dot z}^C V(z) = \max \{ p^T \dot z : p \in \bar \partial V(z) \} $. By the above inequality, we get $D_{\dot z}^C V(z)  \le 0 $ for all $z$ and all $\dot z \in A_\varepsilon z$. Since the Dini derivative is upper bounded by the Clarke derivative \cite{clarke}, we finally get: $D_{\dot z}  V(z) \le D_{\dot z}^C V(z) \le 0 $
 		for all $z$ and all $\dot z \in A_\varepsilon z$.
 		
 		We prove necessity now. For the sake of contradiction, assume that there exists $\ell^*,z$ such that $\nabla V(z) A_{\ell^*} z >0 $. Then, choose $\rho_1,..,\rho_s \ge \varepsilon$ with $\rho_{\ell^*}$ chosen sufficiently large such that $\sum_{\ell} \rho_\ell \nabla V(z) A_\ell z >0$. Then, let $z(t)$ be a trajectory of the LDI with $z(0)=z$ and $\dot z(0)=\sum_{\ell} \rho_\ell A_\ell z \in A_\varepsilon z$. Then, since $\nabla V(z)$ exists, we have $\tfrac d{dt}V(z(0))= \sum_{\ell} \rho_\ell \nabla V(z) A_\ell z>0$; a contradiction.
 	\end{proof}
 	
 	 {	\subsection{Proof of Lemma \ref{lem2}}
 		\begin{proof}
 			Fix $A \in \mathcal{A}_{\varepsilon}$.	Let $\varphi(t;z_0,A)$ be a trajectory of $\dot z(t) = A z(t), z(0)=z_0$. We start with necessity. Since $V$ is non-increasing (in time) then $V(\varphi(t;z_0,A))\le V(z_0)$ for all $z_0$. Since $\varphi(t;z_0,A)=e^{A t} z_0$ and $z_0 \in \mathbb R^N$ is arbitrary we get $V( e^{A t}  z ) \leq V(z)$ for all $z$ as required. \\
 			For sufficiency, we  write $\dot V$ as follows: (where $z(t)=\varphi(t;z_0,A)$)
 			\begin{align*} \dot V(z(t))&= \limsup_{h \to 0^+} \frac{ V(z(t+h))- V(z(t))}{h} \\ &= \limsup_{h \to 0^+} \frac{ V(e^{A h }z(t))- V(z(t))}{h} \le 0,\end{align*}
 			as required.
 	\end{proof}}
 	\subsection{Proof of Lemma \ref{siphonrevisited}}
 	\begin{proof}
 		We show the contrapositive of the result. Take any point $y \in \partial [0,+\infty)^n$, such that $\{ S_i: y_i = 0 \}$ is not a siphon.
 		Hence, there exists $j \in \mathcal{S}$, such that $f_j (y)>0$. Fix $\varepsilon>0$ and $\delta >0$ such that $f_j (x) \geq \delta$ for all $x \in \mathbb{B}_{\varepsilon} (y)$ and $B_{\varepsilon} (y) \cap K = \emptyset$.
 		Denote by $M>0$ any upper bound of $|f(x)|$ in $B_{\varepsilon} (y)$. 
 		Consider any solution $\varphi(t,\xi)$ with $\xi \in K$.
 		If, at any time $t_{\varepsilon/2}$ it enters the ball $B_{\varepsilon/2}$, then by continuity there exists
 		\begin{equation}
 			t_\varepsilon := \max \{ t  \leq t_{\varepsilon/2}: |\varphi(t, \xi) - y|= \varepsilon \}.
 		\end{equation}  
 		Moreover, \begin{align*}\frac \varepsilon 2 \le |\varphi(t_{\varepsilon/2},\xi) - \varphi(t_\varepsilon,\xi) | &= \left |\int_{t_{\varepsilon/2}}^{t_\varepsilon} f(\varphi(\tau,\xi))~d\tau\right | \\ & \le (t_{\varepsilon/2} - t_{\varepsilon}) M. \end{align*}
 		
 		Hence, $(t_{\varepsilon/2} - t_{\varepsilon})\geq \varepsilon / 2M$, and the following holds for the $j$-th component of the solution at time
 		$t_{\varepsilon/2}$:
 		\begin{align*}
 			\varphi_j (t_{\varepsilon/2},\xi) &= \varphi_j (t_{\varepsilon}, \xi )
 			+ \int_{t_{\varepsilon}}^{t_{\varepsilon/2}} f_j( \varphi(\tau,\xi) ) \, d \tau \\ &
 			\geq \varphi_j (t_{\varepsilon}, \xi )
 			+ \delta ( t_{\varepsilon/2} - t_{\varepsilon}  ) \geq \delta \varepsilon / 2M.
 		\end{align*}
 		Moreover, for as long as $\varphi(t,\xi)$ belongs to $B_{\varepsilon}(y)$ we see that
 		the derivative $f_j( \varphi(t,\xi))$ is going to be non-negative.
 		As a consequence, $|\varphi(t,\xi)-y| \geq \min \{ \varepsilon/2, \delta \varepsilon/ 2M \}$, for all $t \geq 0$.
 		This shows that $y \notin \omega(K)$ and concludes the proof of the Lemma.  \end{proof}

 	{ 

 	}
 	 
 \end{document}